\documentclass[11pt,twoside,reqno]{amsart}

\usepackage{amssymb, latexsym}
\usepackage[italian,english]{babel}
\usepackage{amsmath,amsfonts,amsthm}
\usepackage{paralist}
\usepackage[top=3.00cm, bottom=3.00cm, left=2.70cm, right=2.70cm]{geometry}

\usepackage{color}
\numberwithin{equation}{section}
\usepackage{xspace}
\usepackage[bookmarksnumbered,colorlinks]{hyperref}
\usepackage{graphics}
\def\bb#1\eb{\textcolor{blue}
{#1}} %
\def\br#1\er{\textcolor{red}
{#1}} %
\def\bv#1\ev{\textcolor{green}
{#1}} %
\def\bc#1\ec{\textcolor{cyan}
{#1}} %

\usepackage{graphics}

\def\Xint#1{\mathchoice
  {\XXint\displaystyle\textstyle{#1}}%
  {\XXint\textstyle\scriptstyle{#1}}%
  {\XXint\scriptstyle\scriptscriptstyle{#1}}%
  {\XXint\scriptscriptstyle\scriptscriptstyle{#1}}%
  \!\int}
\def\XXint#1#2#3{{\setbox0=\hbox{$#1{#2#3}{\int}$}
  \vcenter{\hbox{$#2#3$}}\kern-.5\wd0}}
\def\-int{\Xint -}

\newcommand{\R}{\mathbb{R}}
\newcommand{\N}{\mathcal{N}}
\newcommand{\h}{H^{s}(\R^{3})}

\DeclareMathOperator{\supp}{supp}
\DeclareMathOperator{\e}{\varepsilon}

\newtheorem{prop}{Proposition}[section]
\newtheorem{lem}{Lemma}[section]
\newtheorem{thm}{Theorem}[section]

\newtheorem{cor}{Corollary}[section]

\newtheorem{remark}{Remark}[section]

\title[Concentration for fractional Schr\"odinger-Kirchhoff equation]{Concentration phenomena for a fractional Schr\"odinger-Kirchhoff type equation}

\author[V. Ambrosio]{Vincenzo Ambrosio}
\address{Vincenzo Ambrosio\hfill\break\indent 
Dipartimento di Scienze Pure e Applicate (DiSPeA),\hfill\break\indent
Universit\`a degli Studi di Urbino `Carlo Bo'\hfill\break\indent
Piazza della Repubblica, 13\hfill\break\indent
61029 Urbino (Pesaro e Urbino, Italy)}
\email{vincenzo.ambrosio@uniurb.it}

\author[T. Isernia]{Teresa Isernia}
\address{Teresa Isernia\hfill\break\indent
Dipartimento di Ingegneria Industriale e Scienze Matematiche \hfill\break\indent
Universit\`a Politecnica delle Marche\hfill\break\indent
Via Brecce Bianche, 1\hfill\break\indent
60131 Ancona (Italy)}
\email{teresa.isernia@unina.it}


\date{}

\begin{document}

\begin{abstract}
In this paper we deal with the multiplicity and concentration of positive solutions for the following fractional Schr\"odinger-Kirchhoff type equation
\begin{equation*}
M\left(\frac{1}{\e^{3-2s}} \iint_{\R^{6}}\frac{|u(x)- u(y)|^{2}}{|x-y|^{3+2s}} dxdy + \frac{1}{\e^{3}} \int_{\R^{3}} V(x)u^{2} dx\right)[\e^{2s} (-\Delta)^{s}u+ V(x)u]= f(u) \, \mbox{ in } \R^{3} 
\end{equation*}
where $\e>0$ is a small parameter, $s\in (\frac{3}{4}, 1)$, $(-\Delta)^{s}$ is the fractional Laplacian, $M$ is a Kirchhoff function, $V$ is a continuous positive potential and $f$ is a superlinear continuous function with subcritical growth.
By using penalization techniques and Ljusternik-Schnirelmann theory, we investigate the relation between the number of positive solutions with the topology of the set where the potential attains its minimum.
\end{abstract}
\maketitle

\section{Introduction}

\noindent
In this paper we are interested in the multiplicity and concentration of positive solutions for the following fractional Schr\"odinger-Kirchhoff type problem
\begin{equation}\label{P}
M\left(\frac{1}{\e^{3-2s}} \iint_{\R^{6}}\frac{|u(x)- u(y)|^{2}}{|x-y|^{3+2s}} dxdy + \frac{1}{\e^{3}} \int_{\R^{3}} V(x)u^{2} dx\right) \Bigl[\e^{2s} (-\Delta)^{s}u+ V(x)u\Bigr]= f(u) \, \mbox{ in } \R^{3} 
\end{equation}
where $\e>0$ is a small parameter, $s\in (\frac{3}{4}, 1)$, $(-\Delta)^{s}$ denotes the usual fractional Laplacian operator, $M: [0, +\infty)\rightarrow [0, +\infty)$ is the Kirchhoff term, $V:\R^{3}\rightarrow \R$ and $f:\R\rightarrow \R$ are continuous functions satisfying suitable assumptions.  \\
When $\e=1$ and $V(x)\equiv 0$, we obtain the fractional stationary Kirchhoff equation 
\begin{equation}\label{fKe}
M\left( \iint_{\R^{6}}\frac{|u(x)- u(y)|^{2}}{|x-y|^{3+2s}} dxdy \right)\, (-\Delta)^{s}u= f(u) \, \mbox{ in } \R^{3} 
\end{equation}
which has been introduced for the first time by Fiscella \& Valdinoci in \cite{FV} (in the case of bounded domains) and extensively 
studied in the last years by many authors; see for instance \cite{AI, AFP, FMS, MV, PuSa, PXZ} and references therein for several existence and multiplicity results in any dimension, in the whole space and in bounded domains.  \\
We recall that the local counterpart of \eqref{fKe} is related to the famous Kirchhoff equation
\begin{equation}\label{Ke}
\rho \, u_{tt} - \left( \frac{P_{0}}{h}+ \frac{E}{2L}\int_{0}^{L} |u_{x}|^{2} dx \right) u_{xx} =0
\end{equation}
introduced by Kirchhoff \cite{K} in $1883$ as a nonlinear extension of D'Alembert' s wave equation for free vibrations of elastic strings. Here $u=u(x, t)$ is the transverse string displacement at the space coordinate $x$ ant time $t$, $L$ is the length of the string, $h$ is the area of the cross section, $E$ is Young's modulus of the material, $\rho$ is the mass density, and $P_{0}$ is the initial tension.  \\
The early investigations dedicated to the Kirchhoff equation \eqref{Ke} were given by Bernstein \cite{B} and Pohozaev \cite{P}. Anyway, Kirchhoff equation \eqref{Ke} began to call attention of several researchers only after the work of Lions \cite{Lions}, where a functional analysis approach was introduced to attack it. For more details on classical Kirchhoff problems, we refer to \cite{ACF, AP, APS, CKW, Fsantos, PZ, WTXZ}.  \\
In a recent paper \cite{FV}, Fiscella \& Valdinoci have proposed an interesting physical interpretation of Kirchhoff equation in the fractional scenario. In their correction of the early (one-dimensional) model, the tension on the string, which classically has a ``nonlocal" nature arising from the average of the kinetic energy $\frac{|u_{x}|^{2}}{2}$ on $[0, L]$, possesses a further nonlocal behavior provided by the $H^{s}$-norm (or other more general fractional norms) of the function $u$.\\
On the other hand, when $M=1$, \eqref{P} becomes the time dependent fractional Schr\"odinger equation 
\begin{equation}\label{fSe}
\e^{2s} (-\Delta)^{s}u+ V(x)u= f(u) \, \mbox{ in } \R^{3} 
\end{equation} 
which plays a fundamental role in fractional quantum mechanic; see \cite{DDPDV, DDPW, Laskin1, Laskin2} for a physical interpretation. 
Equation \eqref{fSe} can be seen as the fractional analogue of the celebrated Schr\"odinger equation
\begin{equation}\label{Se}
-\e^{2s} \Delta u+ V(x)u= f(x, u) \, \mbox{ in } \R^{N} 
\end{equation} 
which has been widely investigated in the last two decades. Since we cannot review the huge bibliography of \eqref{Se}, we just cite \cite{AMS, DF, FW, Rab, Willem} and references therein.  \\
In the last years, the concentration of positive solutions to \eqref{fSe} has attracted the attention of many mathematicians \cite{AM, A, A2, CZ, DDPW, FFV, HZ}. 
In particular, in \cite{AM} Alves \& Miyagaki used the penalization method to study the concentration phenomenon of positive solutions for fractional Schr\"odinger equation \eqref{fSe} when $V$ has a local minimum and $f$ is subcritical. He \& Zou \cite{HZ} investigated the relation between the number of positive solutions of \eqref{fSe} with $f(u)= g(u)+ u^{2^{*}_{s}-1}$, where $g$ is subcritical, 
and the topology of the set where the potential $V$ attains its minima. In \cite{A} the first author complemented the results in \cite{AM} and \cite{HZ} dealing with the multiplicity and concentration of solutions in the subcritical and supercritical cases. \\
Motivated by the above papers, in this work we focus our attention on the multiplicity and the concentration behavior of positive solutions to the fractional Schr\"odinger-Kirchhoff type problem \eqref{P}. To our knowledge, this type of investigation has not ever been done in fractional setting when $M$ is not constant. The aim of this paper is to fill this gap. \\
Before to state our result, we introduce the main assumptions. 
Along the paper we assume that $M:\R^{+}\rightarrow \R^{+}$ is a continuous function satisfying 
\begin{compactenum}[$(M_1)$]
\item there exists $m_{0}>0$ such that $M(t)\geq m_{0}$ for any $t\geq 0$; 
\item the function $t\mapsto M(t)$ is increasing; 
\item for each $t_{1}\geq t_{2}>0$ it holds
\begin{equation*}
\frac{M(t_1)}{t_1}-\frac{M(t_2)}{t_2}\leq m_{0} \left( \frac{1}{t_1}- \frac{1}{t_{2}}\right).  
\end{equation*}
\end{compactenum}
As a model for $M$, we can take $M(t)= m_{0}+ bt + \sum_{i=1}^{k} b_{i}t^{\gamma_{i}}$ with $b_{i}\geq 0$ and $\gamma_{i}\in (0, 1)$ for all $i\in \{1, \dots , k\}$. \\
On the potential $V:\R^{3}\rightarrow \R$, we suppose that $V\in C(\R^{3}, \R)$ and verifies the following hypotheses:
\begin{compactenum}[$(V_1)$]
\item there exists $V_{0}>0$ such that $\displaystyle{V_{0}:= \inf_{x\in \R^{3}} V(x)}$; 
\item for each $\delta>0$ there is a bounded and Lipschitz domain $\Omega\subset \R^{3}$ such that
\begin{equation*}
V_{0}< \min_{\partial \Omega} V, \quad \Lambda= \{x\in \Omega : V(x)= V_{0}\} \neq \emptyset 
\end{equation*}
and
\begin{equation*}
\Lambda_{\delta}= \{x \in \R^{3} : {\rm dist}(x, \Lambda)\leq \delta \}\subset \Omega. 
\end{equation*}
\end{compactenum}
Concerning the nonlinear term in \eqref{P}, we assume that $f: \R \rightarrow \R$ is a continuous function satisfying the following conditions: 
\begin{compactenum}[$(f_1)$]
\item $\displaystyle{\lim_{t\rightarrow 0^{+}}\frac{f(t)}{t^{3}}=0}$; 
\item there is $q\in (4, \frac{6}{3-2s})$ such that $\displaystyle{\lim_{t\rightarrow \infty} \frac{f(t)}{t^{q-1}}=0}$; 
\item there is $\vartheta \in (4, \frac{6}{3-2s})$ such that $0<\vartheta F(t)\leq f(t)t$ for any $t>0$; 
\item the function $\displaystyle{t\mapsto \frac{f(t)}{t^{3}}}$ is non-decreasing in $(0, \infty)$. 
\end{compactenum}
A typical example of $f$ is given by 
$$
f(t)=\sum_{i=1}^{k} a_{i} (t^{+})^{q_{i}-1}
$$
with $a_{i}\geq 0$ not all identically zero and $q_{i}\in [\vartheta, \frac{6}{3-2s})$ for all $i\in \{1, \dots, k\}$.\\
We note that the assumption $(f_4)$ implies that 
\begin{equation}\label{rem1}
t\mapsto \frac{1}{4}f(t)t- F(t)  \mbox{ is nondecreasing for any } t\geq 0.
\end{equation}
Since we are interested in positive solutions, we assume that $f$ vanishes in $(-\infty, 0)$. \\
Now, we are ready to state our main result.
\begin{thm}\label{thm1}
Let $s\in (\frac{3}{4}, 1)$ and assume that $(M_{1})$-$(M_{3})$, $(V_{1})$- $(V_{2})$ and $(f_{1})$-$(f_{4})$ hold true. Then, given $\delta>0$ there is $\bar{\e}= \bar{\e}(\delta)>0$ such that the problem \eqref{P} has at least $cat_{\Lambda_{\delta}}(\Lambda)$ positive solutions, for all $\e \in (0, \bar{\e})$. Moreover, if $u_{\e}$ denotes one of these positive solutions and $\eta_{\e}\in \R^{3}$ its global maximum, then
\begin{equation*}
\lim_{\e\rightarrow 0} V(\eta_{\e})=V_{0}. 
\end{equation*}
\end{thm}

\noindent
We recall that if $Y$ is a given closed set of a topological space $X$, we denote by $cat_{Y}(Y)$ the Ljusternik-Schnirelmann category of $Y$ in $X$, that is the least number of closed and contractible sets in $X$ which cover $Y$; see \cite{Willem}. \\
The proof of Theorem \ref{thm1} relies on variational methods developed in classical framework in \cite{Fsantos}. 
Clearly, the presence of the fractional Laplacian makes our analysis more delicate and intriguing with respect to the one performed in local setting, and the recent results obtained in \cite{AM, FQT} to study fractional Schr\"odinger equations will have a fundamental role to overcome our difficulties. \\
In what follows, we give a sketch of the proof. The lack of informations on the behavior of $V$ at infinity suggest us to use the penalization method introduced by Del Pino \& Felmer \cite{DF}. Since $f$ and $M$ are only continuous, the Nehari manifold associated to the modified problem is not differentiable, so the well-known arguments on the Nehari manifold do not work in our setting. To circumvent this obstacle, we will use some abstract results due to Szulkin \& Weth in \cite{SW}. After a careful study of the autonomous problem associated to \eqref{P}, we deal with the multiplicity of solutions of the modified problem, by invoking the Ljusternik-Schnirelmann theory. Then, in order to prove that the solutions $u_{\e}$ of the truncated problem are also solutions to \eqref{P} when $\e>0$ is sufficiently small, we argue as in \cite{AM}, providing $L^{\infty}$ estimates for $u_{\e}$-adapting the Moser's iteration \cite{Moser} in nonlocal framework- and by using some useful properties of the Bessel kernels established in \cite{FQT}. 
We point out that the restriction $s\in (\frac{3}{4}, 1)$ is essential in our technical approach in order to guarantee the  embedding of the space $H^{s}(\R^{3})$ into the Lebesgue spaces $L^{r}(\R^{N})$ with $4\leq r< \frac{6}{3-2s}$ (see conditions $(f_1)$-$(f_3)$).\\
Finally, we would like to emphasize that Theorem \ref{thm1} corresponds to the nonlocal counterpart of Theorem $1.1$ in \cite{Fsantos}. 
As far as we know the results presented here are new in literature. 
\smallskip

\noindent
The plan of the paper is the following. In Section $2$ we give some useful results related to the fractional Sobolev spaces. In Section $3$ we truncate the nonlinearity and we show that the modified problem admits a positive solution. In Section $4$ we study the autonomous problem associated to \eqref{P}. In Section $5$, we introduce the barycenter map and its properties. This tool will be crucial to obtain a multiplicity result for the modified problem via the abstract category theory of Ljusternik-Schnirelmann.
The last Section is devoted to the proof of Theorem \ref{thm1}.

\section{Fractional Sobolev spaces}
\noindent
In this section we offer a rather sketchy review of the fractional Sobolev spaces and some useful results which will be used later.  For more details, we refer to \cite{BV, DPV, DMV, MBRS, SV1, SV2}.\\
Fix $s\in (0,1)$. The fractional Laplacian $(-\Delta)^{s}$ is a pseudo-differential operator defined via Fourier transform by 
$$
\mathcal{F}(-\Delta)^{s}u(k)=|k|^{2s} \mathcal{F}u(k) \quad (k\in \R^{N}),
$$
when $u: \R^{N}\rightarrow \R$ belongs to the Schwarz space of rapidly decaying $C^{\infty}$ functions in $\R^{N}$.
Equivalently, $(-\Delta)^{s}$ can be represented as
$$
(-\Delta)^{s}u(x)=-\frac{C_{N,s}}{2} \int_{\R^{N}} \frac{u(x+y)+u(x-y)-2u(x)}{|y|^{N+2s}} \,dy  \quad (x\in \R^{N})
$$
where $C_{N,s}$ is a dimensional constant depending only on $N$ and $s$; see \cite{DPV} for more details.

\noindent
Let us denote by $\mathcal{D}^{s, 2}(\R^{N})$ the completion of $C^{\infty}_{c}(\R^{N})$ with respect to the Gagliardo (semi) norm 
$$
[u]^{2}=\int_{\R^{N}} |(-\Delta)^{\frac{s}{2}} u|^{2} \, dx= \iint_{\R^{2N}} \frac{|u(x)-u(y)|^{2}}{|x-y|^{N+2s}}\, dxdy.
$$
Now, we define the fractional Sobolev space
$$
H^{s}(\R^{N})= \left\{u\in L^{2}(\R^{N}) : [u]<\infty \right \}
$$
endowed with the natural norm 
$$
\|u\|_{H^{s}(\R^{N})} = \sqrt{[u]^{2} + \int_{\R^{N}} |u|^{2} \,dx}.
$$

\noindent
We recall the following embeddings of the fractional Sobolev spaces into Lebesgue spaces.
\begin{thm}\cite{DPV}\label{Sembedding}
Let $s\in (0,1)$ and $N>2s$. Then there exists a sharp constant $S_{*}=S(N, s)>0$
such that for any $u\in \mathcal{D}^{s, 2}(\R^{N})$
\begin{equation*}
\|u\|^{2}_{L^{2^{*}_{s}}(\R^{N})} \leq S_{*} [u]^{2}. 
\end{equation*}
Moreover $H^{s}(\R^{N})$ is continuously embedded in $L^{p}(\R^{N})$ for any $p\in [2, 2^{*}_{s}]$ and compactly in $L^{p}_{loc}(\R^{N})$ for any $p\in [1, 2^{*}_{s})$. 
\end{thm}

\noindent
The following lemma is a version of the well-known concentration-compactness principle:
\begin{lem}\cite{Secchi}\label{Lions}
Let $N>2s$. If $\{u_{n}\}_{n\in \mathbb{N}}$ is a bounded sequence in $H^{s}(\R^{N})$ and if
$$
\lim_{n \rightarrow \infty} \sup_{y\in \R^{N}} \int_{B_{R}(y)} |u_{n}|^{2} dx=0
$$
where $R>0$,
then $u_{n}\rightarrow 0$ in $L^{r}(\R^{N})$ for all $r\in (2, 2^{*}_{s})$.
\end{lem}

\noindent
We also have the following useful result.
\begin{lem}\cite{PP}\label{pp}
Assume that $N>2s$ and $u\in \mathcal{D}^{s, 2}(\R^{N})$. Let $\varphi\in C^{\infty}_{c}(\R^{N})$ and for each $r>0$ we define $\varphi_{r}(x)=\varphi(x/r)$. Then, $\int_{\R^{N}} |(-\Delta)^{\frac{s}{2}}(u \varphi_{r})|^{2}\,dx \rightarrow 0$ as $r\rightarrow 0$.
If in addition $\varphi=1$ in a neighborhood of the origin, then
$\int_{\R^{N}} |(-\Delta)^{\frac{s}{2}}(u \varphi_{r})|^{2}\,dx \rightarrow \int_{\R^{N}} |(-\Delta)^{\frac{s}{2}}u|^{2}\,dx$ as $r\rightarrow \infty$.
\end{lem}

\section{The modified problem}
\noindent
This section is devoted to the existence of positive solutions to \eqref{P}. From now on, we assume $N=3$ and $s\in (\frac{3}{4}, 1)$. 
After a change of variable, the problem \eqref{P} reduces to
\begin{equation}\label{Pe}
M\left(\iint_{\R^{6}}\frac{|u(x)- u(y)|^{2}}{|x-y|^{3+2s}} dxdy + \int_{\R^{3}} V(\e x)u^{2} dx\right)[(-\Delta)^{s}u+ V(\e x)u]= f(u) \, \mbox{ in } \R^{3}.  
\end{equation}
Take $K>\frac{2}{m_{0}}$ and $a>0$ such that $f(a)= \frac{V_{0}}{K}a$. Let us define
\begin{equation*}
\tilde{f}(t)=
\left\{
\begin{array}{ll}
f(t) & \mbox{ if } t\leq a\\
\frac{V_{0}}{K}t & \mbox{ if } t>a
\end{array}
\right.
\end{equation*}
and 
\begin{equation*}
g(x, t) = \chi_{\Omega}(x) f(t) + (1- \chi_{\Omega}(x))\tilde{f}(t). 
\end{equation*}
From the assumptions on $f$ we deduce that $g$ is a Carath\'eodory function and satisfies
\begin{compactenum}[$(g_1)$]
\item $\displaystyle{\lim_{t\rightarrow 0^{+}} \frac{g(x,t)}{t^{3}}=0}$ uniformly in $x\in \R^{3}$; 
\item $\displaystyle{\lim_{t\rightarrow \infty} \frac{g(x,t)}{t^{q-1}}=0}$ uniformly in $x\in \R^{3}$; 
\item (i) $0\displaystyle{\leq \vartheta G(x,t) < g(x,t) t}$ for any $x\in \Omega$ and for any $t>0$, \\
(ii) $\displaystyle{0\leq 2 G(x,t) \leq g(x,t) t \leq \frac{V_{0}}{K}t^{2}}$ for any $x\in \R^{3}\setminus \Omega$ and for any $t>0$; 
\item for each $x\in \Omega$ the application $\displaystyle{t\mapsto \frac{g(x,t)}{t^{3}}}$ is increasing in $(0, \infty)$ and for each $x\in \R^{3}\setminus \Omega$ the application $\displaystyle{t\mapsto \frac{g(x,t)}{t^{3}}}$ is increasing in $(0, a)$. 
\end{compactenum}
From the definition of $g$ follows that
\begin{align*}
&g(x, t)\leq f(t)  \quad \mbox{ for all } t>0, \mbox{ for all } x\in \R^{3}, \\
&g(x, t)=0  \quad \mbox{ for all } t<0, \mbox{ for all } x\in \R^{3}. 
\end{align*}
In what follows, we consider the auxiliary problem 
\begin{equation}\label{Pea}
M\left(\iint_{\R^{6}}\frac{|u(x)- u(y)|^{2}}{|x-y|^{3+2s}} dxdy + \int_{\R^{3}} V(\e x)u^{2} dx\right)[(-\Delta)^{s}u+ V(\e x)u]= g(\e x, u) \, \mbox{ in } \R^{3}.  
\end{equation}
Moreover, we focus our attention on positive solutions to \eqref{Pea} with $u(x)\leq a$ for each $x\in \R^{3}\setminus \Omega$.
Indeed, from definitions of $g$, it is clear that solutions having the above property are also solutions to the starting problem \eqref{Pe}.

\noindent
Therefore, solutions of \eqref{Pea} can be found as critical points of the following energy functional 
\begin{equation*}
\mathcal{J}_{\e}(u)= \frac{1}{2} \widehat{M}\left( \iint_{\R^{6}}\frac{|u(x)- u(y)|^{2}}{|x-y|^{3+2s}} dxdy + \int_{\R^{3}} V(\e x)u^{2} dx\right) - \int_{\R^{3}} G(\e x, u)\, dx
\end{equation*}
where 
\begin{equation*}
\widehat{M}(t)= \int_{0}^{t} M(\tau)\, d\tau \quad \mbox{ and } \quad G(\e x, t)= \int_{0}^{t} g(\e x, \tau)\, d\tau, 
\end{equation*}
which is well defined on the Hilbert space 
\begin{equation*}
\mathcal{H}_{\e}= \left\{u\in \h : \int_{\R^{3}} V(\e x)\, u^{2} dx<\infty \right\}
\end{equation*}
endowed with the inner product
\begin{equation*}
(u, \varphi)_{\e} = \iint_{\R^{6}} \frac{(u(x)- u(y))(\varphi(x) - \varphi(y))}{|x-y|^{3+2s}} \,dxdy + \int_{\R^{3}} V(\e x)\, u(x)\, \varphi(x) dx. 
\end{equation*}
The norm induced by the inner product is given by
\begin{equation*}
\|u\|_{\e}^{2} = \iint_{\R^{6}}\frac{|u(x)- u(y)|^{2}}{|x-y|^{3+2s}} \,dxdy + \int_{\R^{3}} V(\e x)\, u^{2} \,dx. 
\end{equation*}

\noindent
From the assumptions on $M$ and $f$, and by using Theorem \ref{Sembedding}, it is easy to check that $\mathcal{J}_{\e}$ is well-defined, $\mathcal{J}_{\e} \in C^{1}(\mathcal{H}_{\e}, \R)$ and that its differential $\mathcal{J}'$ is given by 
\begin{equation*}
\langle \mathcal{J}_{\e}'(u), \varphi \rangle = M(\|u\|_{\e}^{2}) (u, \varphi)_{\e} - \int_{\R^{3}} g(\e x, u)\varphi \, dx,  
\end{equation*}
for any $u, \varphi \in \mathcal{H}_{\e}$.
Let us introduce the Nehari manifold associated to $\mathcal{J}_{\e}$, that is 
\begin{equation*}
\N_{\e}= \{u\in \mathcal{H}_{\e}\setminus \{0\} : \langle \mathcal{J}_{\e}'(u), u \rangle=0\}. 
\end{equation*}

\noindent
The main result of this Section is the following.
\begin{thm}\label{thm2.1}
Under the assumptions $(M_{1})$-$(M_{3})$, $(V_{1})$-$(V_{2})$ and $(f_{1})$-$(f_{4})$, the auxiliary problem \eqref{Pea} has a nonnegative ground state solution for all $\e>0$. 
\end{thm} 

\noindent
We denote by $\Omega_{\e}= \{x\in \R^{3} : \e x\in \Omega\}$ and 
\begin{equation*}
\mathcal{H}_{\e}^{+}= \{u\in \mathcal{H}_{\e} : |\supp(u^{+}) \cap \Omega_{\e}|>0\}\subset \mathcal{H}_{\e}. 
\end{equation*}
Let $\mathbb{S}_{\e}$ be the unit sphere of $\mathcal{H}_{\e}$ and we denote by $\mathbb{S}_{\e}^{+}= \mathbb{S}_{\e}\cap \mathcal{H}_{\e}^{+}$.
We observe that $\mathcal{H}_{\e}^{+}$ is open in $\mathcal{H}_{\e}$.
Indeed, let us consider a sequence $\{u_{n}\}_{n\in \mathbb{N}}\subset \mathcal{H}_{\e} \setminus \mathcal{H}_{\e}^{+}$ such that $u_{n}\rightarrow u$ in $\mathcal{H}_{\e}$ and assume by contradiction that $u\in \mathcal{H}_{\e}^{+}$. Now, from the definition of $\mathcal{H}_{\e}^{+}$ it follows that $|\supp(u_{n}^{+})\cap \Omega_{\e}|=0$ for all $n\in \mathbb{N}$ and $u_{n}^{+}(x)\rightarrow u^{+}(x)$ a.e. in $x\in \Omega_{\e}$. So,
\begin{equation*}
u^{+}(x)= \lim_{n\rightarrow \infty} u_{n}^{+}(x)=0 \, \mbox{ a.e. in } x\in \Omega_{\e}, 
\end{equation*}
and this contradicts the fact that $u\in \mathcal{H}_{\e}^{+}$. Therefore $\mathcal{H}_{\e}^{+}$ is open.

\noindent
From the definition of $\mathbb{S}_{\e}^{+}$ and the fact that $\mathcal{H}_{\e}^{+}$ is open in $\mathcal{H}_{\e}$, it follows that $\mathbb{S}_{\e}^{+}$ is a incomplete $C^{1,1}$-manifold of codimension $1$, modeled on $\mathcal{H}_{\e}$ and contained in the open $\mathcal{H}_{\e}^{+}$. Thus $\mathcal{H}_{\e}=T_{u} \mathbb{S}_{\e}^{+} \oplus \R u$ for each $u\in \mathbb{S}_{\e}^{+}$, where
\begin{equation*}
T_{u} \mathbb{S}_{\e}^{+}= \{v \in \mathcal{H}_{\e} : (u, \varphi)_{\e}=0\}.
\end{equation*} 
In the next lemma we prove that $\mathcal{J}_{\e}$ has a mountain pass geometry. 
\begin{lem}\label{lem2.2}
The functional $\mathcal{J}_{\e}$ satisfies the following
\begin{compactenum}[$(a)$]
\item there exist $\alpha, \rho>0$ such that $\mathcal{J}_{\e}(u) \geq \alpha $ with $\|u\|_{\e}= \rho$; 
\item there exists $e\in \mathcal{H}_{\e}\setminus B_{\rho}(0)$ such that $\mathcal{J}_{\e}(e)<0$.
\end{compactenum}
\end{lem}

\begin{proof}
$(a)$ From the assumptions $(M_1)$, $(g_1)$, $(g_2)$ and Theorem \ref{Sembedding}, it follows that for any $\xi>0$ 
\begin{equation*}
\mathcal{J}_{\e}(u)=\frac{1}{2}\widehat{M}(\|u\|_{\e}^{2})- \int_{\R^{3}} G(\e x, u)\, dx \geq \frac{m_{0}}{2} \|u\|_{\e}^{2} -  \xi C\|u\|_{\e}^{4}- C_{\xi}C\|u\|_{\e}^{q}. 
\end{equation*}
Thus, we can find $\alpha, \rho>0$ such that $\mathcal{J}_{\e}(u) \geq \alpha $ with $\|u\|_{\e}= \rho$.\\
$(b)$ By using the assumption $(M_3)$ we can infer that there exists a positive constant $\gamma$ such that 
\begin{equation}\label{cc}
M(t)\leq \gamma (1+t) \, \mbox{ for all } \, t\geq 0.
\end{equation} 
Then, in view of $(g_3)$-(i), we can see that, for any $u\in \mathcal{H}_{\e}^{+}$ and $t>0$
\begin{align}\label{b2.2}
\mathcal{J}_{\e}(tu)&= \frac{1}{2} \widehat{M}(\|tu\|_{\e}^{2})- \int_{\R^{3}} G(\e x, tu)\, dx\nonumber\\
&\leq \frac{\gamma}{2}t^{2} \|u\|_{\e}^{2}+ \frac{\gamma}{4}t^{4} \|u\|_{\e}^{4} -\int_{\Omega_{\e}} G(\e x, tu)\, dx\nonumber\\
&\leq \frac{\gamma}{2}t^{2} \|u\|_{\e}^{2}+ \frac{\gamma}{4}t^{4} \|u\|_{\e}^{4} - C_{1} t^{\vartheta}\int_{\Omega_{\e}} (u^{+})^{\vartheta} \, dx + C_{2} |\supp(u^{+})\cap \Omega_{\e}|, 
\end{align}
for some positive constants $C_1$ and $C_2$. \\
Taking into account that $\vartheta \in (4,\frac{6}{3-2s})$, we get $\mathcal{J}_{\e}(t u)\rightarrow -\infty \mbox{ as } t\rightarrow +\infty$. 
\end{proof}

\noindent
Since $f$ and $M$ are continuous functions, we need the next results to overcome the non-differentiability of $\N_{\e}$ and the incompleteness of $\mathbb{S}_{\e}^{+}$.
\begin{lem}\label{lem2.3}
Assume that $(M_1)-(M_3)$, $(V_1)-(V_2)$ and $(f_1)-(f_4)$ hold true. Then, 
\begin{compactenum}[$(i)$]
\item For each $u\in \mathcal{H}_{\e}^{+}$, let $h:\R^{+}\rightarrow \R$ be defined by $h_{u}(t)= \mathcal{J}_{\e}(tu)$. Then, there is a unique $t_{u}>0$ such that 
\begin{align*}
&h'_{u}(t)>0 \mbox{ in } (0, t_{u})\\
&h'_{u}(t)<0 \mbox{ in } (t_{u}, \infty); 
\end{align*}
\item there exists $\tau>0$ independent of $u$ such that $t_{u}\geq \tau$ for any $u\in \mathbb{S}_{\e}^{+}$. Moreover, for each compact set $\mathbb{K}\subset \mathbb{S}_{\e}^{+}$ there is a positive constant $C_{\mathbb{K}}$ such that $t_{u}\leq C_{\mathbb{K}}$ for any $u\in \mathbb{K}$; 
\item The map $\hat{m}_{\e}: \mathcal{H}_{\e}^{+}\rightarrow \N_{\e}$ given by $\hat{m}_{\e}(u)= t_{u}u$ is continuous and $m_{\e}:= \hat{m}_{\e}|_{\mathbb{S}_{\e}^{+}}$ is a homeomorphism between $\mathbb{S}_{\e}^{+}$ and $\N_{\e}$. Moreover $m_{\e}^{-1}(u)=\frac{u}{\|u\|_{\e}}$; 
\item If there is a sequence $\{u_{n}\}_{n\in \mathbb{N}}\subset \mathbb{S}_{\e}^{+}$ such that ${\rm dist}(u_{n}, \partial \mathbb{S}_{\e}^{+})\rightarrow 0$ then $\|m_{\e}(u_{n})\|_{\e}\rightarrow \infty$ and $\mathcal{J}_{\e}(m_{\e}(u_{n}))\rightarrow \infty$.
\end{compactenum}

\end{lem}

\begin{proof}
$(i)$ We know that $h_{u}\in C^{1}(\R^{+}, \R)$, and by Lemma \ref{lem2.2} we have that $h_{u}(0)=0$, $h_{u}(t)>0$ for $t>0$ small enough and $h_{u}(t)<0$ for $t>0$ sufficiently  large. So there exists $t_{u}>0$ such that $h_{u}'(t_{u})=0$, and $t_{u}$ is a global maximum for $h_{u}$. \\
Then,  
\begin{equation*}
0= h_{u}'(t_{u})= \langle \mathcal{J}_{\e}'(t_{u}u), u\rangle = \frac{1}{t_{u}} \langle \mathcal{J}_{\e}'(t_{u}u), t_{u}u\rangle
\end{equation*}
from which we deduce that $t_{u}u \in \N_{\e}$. \\
Now, we aim to prove the uniqueness of such $t_{u}$. Assume by contradiction that there exist $t_{1}> t_{2}>0$ such that $h_{u}'(t_{1})=h_{u}'(t_{2})=0$, or equivalently
\begin{align}
&t_{1} M(\|t_{1}u \|_{\e}^{2}) \|u\|_{\e}^{2} = \int_{\R^{3}} g(\e x, t_{1}u) u\, dx \label{t_1}\\
&t_{2} M(\|t_{2}u \|_{\e}^{2}) \|u\|_{\e}^{2} = \int_{\R^{3}} g(\e x, t_{2}u) u\, dx. \label{t_2}
\end{align}
Dividing both members of \eqref{t_1} by $t_{1}^{3}\|u\|_{\e}^{4}$ we get
\begin{align*}
\frac{M(\|t_{1}u \|_{\e}^{2})}{\|t_{1}u\|_{\e}^{2}} = \frac{1}{\|u\|_{\e}^{4}}\int_{\R^{3}} \frac{g(\e x, t_{1}u)}{(t_{1}u)^{3}} u^{4} dx, 
\end{align*}
similarly, dividing both members of \eqref{t_2} by $t_{2}^{3}\|u\|_{\e}^{4}$ we obtain
\begin{align*}
\frac{M(\|t_{2}u \|_{\e}^{2})}{\|t_{2}u\|_{\e}^{2}} = \frac{1}{\|u\|_{\e}^{4}}\int_{\R^{3}} \frac{g(\e x, t_{2}u)}{(t_{2}u)^{3}} u^{4} dx.  
\end{align*}
Subtracting the above identities, and taking into account $(M_3)$ and $(g_4)$ we can see that  
\begin{align*}
\frac{m_{0}}{\|u\|_{\e}^{2}}\left(\frac{1}{t_{1}^{2}}-\frac{1}{t_{2}^{2}}\right)&\geq \frac{M(\|t_{1}u \|_{\e}^{2})}{\|t_{1}u\|_{\e}^{2}}-\frac{M(\|t_{2}u \|_{\e}^{2})}{\|t_{2}u\|_{\e}^{2}}\\
&= \frac{1}{\|u\|_{\e}^{4}}\int_{\R^{3}} \left[\frac{g(\e x, t_{1}u)}{(t_{1}u)^{3}} - \frac{g(\e x, t_{2}u)}{(t_{2}u)^{3}}\right] u^{4} dx\\
&=\frac{1}{\|u\|_{\e}^{4}}\int_{\R^{3}\setminus \Omega_{\e}} \left[\frac{g(\e x, t_{1}u)}{(t_{1}u)^{3}} - \frac{g(\e x, t_{2}u)}{(t_{2}u)^{3}}\right] u^{4} dx+ \frac{1}{\|u\|_{\e}^{4}}\int_{\Omega_{\e}} \left[\frac{g(\e x, t_{1}u)}{(t_{1}u)^{3}} - \frac{g(\e x, t_{2}u)}{(t_{2}u)^{3}}\right] u^{4} dx\\
&\geq \frac{1}{\|u\|_{\e}^{4}}\int_{\R^{3}\setminus \Omega_{\e}} \left[\frac{g(\e x, t_{1}u)}{(t_{1}u)^{3}} - \frac{g(\e x, t_{2}u)}{(t_{2}u)^{3}}\right] u^{4} dx\\
&= \frac{1}{\|u\|_{\e}^{4}}\int_{(\R^{3}\setminus \Omega_{\e})\cap \{t_{2}u>a\}} \left[\frac{g(\e x, t_{1}u)}{(t_{1}u)^{3}} - \frac{g(\e x, t_{2}u)}{(t_{2}u)^{3}}\right] u^{4} dx \\
&+ \frac{1}{\|u\|_{\e}^{4}}\int_{(\R^{3}\setminus \Omega_{\e})\cap \{t_{2}u\leq a <t_{1}u\}} \left[\frac{g(\e x, t_{1}u)}{(t_{1}u)^{3}} - \frac{g(\e x, t_{2}u)}{(t_{2}u)^{3}}\right] u^{4} dx\\
&+\frac{1}{\|u\|_{\e}^{4}}\int_{(\R^{3}\setminus \Omega_{\e})\cap \{t_{1}u<a\}} \left[\frac{g(\e x, t_{1}u)}{(t_{1}u)^{3}} - \frac{g(\e x, t_{2}u)}{(t_{2}u)^{3}}\right] u^{4} dx =:I+II+III
\end{align*}
Let us observe that $III\geq 0$ in view of $(g_4)$ and $t_1>t_2$. Taking into account the definition of $g$, we have 
\begin{align*}
I&\geq \frac{1}{\|u\|_{\e}^{4}} \int_{(\R^{3}\setminus \Omega_{\e})\cap \{t_{2}u>a\}} \left[ \frac{V_{0}}{K} \frac{1}{(t_{1}u)^{2}} - \frac{V_{0}}{K}\frac{1}{(t_{2}u)^{2}}\right]u^{4}dx \\
&= \frac{1}{\|u\|_{\e}^{4}}\frac{1}{K}\left(\frac{1}{t_{1}^{2}} - \frac{1}{t_{2}^{2}}\right) \int_{(\R^{3}\setminus \Omega_{\e})\cap \{t_{2}u>a\}} V_{0}u^{2} dx. 
\end{align*}
Regarding $II$, from the definition of $g$ and taking into account that $g(x,t)\leq f(t)$, we can infer
\begin{align*}
II&\geq \frac{1}{\|u\|_{\e}^{4}}\int_{(\R^{3}\setminus \Omega_{\e})\cap \{t_{2}u\leq a <t_{1}u\}} \left[ \frac{V_{0}}{K}\frac{1}{(t_{1}u)^{2}} - \frac{f(t_{2}u)}{(t_{2}u)^{3}}\right] u^{4} dx. 
\end{align*}
Therefore we deduce
\begin{align*}
\frac{m_{0}}{\|u\|_{\e}^{2}}\left(\frac{1}{t_{1}^{2}}-\frac{1}{t_{2}^{2}}\right) &\geq \frac{1}{\|u\|_{\e}^{4}}\frac{1}{K}\left(\frac{1}{t_{1}^{2}} - \frac{1}{t_{2}^{2}}\right) \int_{(\R^{3}\setminus \Omega_{\e})\cap \{t_{2}u>a\}} V_{0}u^{2} dx \\
&+ \frac{1}{\|u\|_{\e}^{4}}\int_{(\R^{3}\setminus \Omega_{\e})\cap \{t_{2}u\leq a <t_{1}u\}} \left[ \frac{V_{0}}{K}\frac{1}{(t_{1}u)^{2}} - \frac{f(t_{2}u)}{(t_{2}u)^{3}}\right] u^{4} dx.  
\end{align*}
Multiplying both sides by $\|u\|_{\e}^{4}\frac{t_{1}^{2}t_{2}^{2}}{t_{2}^{2}- t_{1}^{2}}<0$, we have
\begin{align*}
m_{0}\|u\|_{\e}^{2}&\leq \frac{1}{K} \int_{(\R^{3}\setminus \Omega_{\e})\cap \{t_{2}u>a\}} V_{0}u^{2} dx + \frac{t_{1}^{2}t_{2}^{2}}{t_{2}^{2}- t_{1}^{2}} \int_{(\R^{3}\setminus \Omega_{\e})\cap \{t_{2}u\leq a <t_{1}u\}} \left[ \frac{V_{0}}{K}\frac{1}{(t_{1}u)^{2}} - \frac{f(t_{2}u)}{(t_{2}u)^{3}}\right] u^{4} dx \\
&= \frac{1}{K} \int_{(\R^{3}\setminus \Omega_{\e})\cap \{t_{2}u>a\}} V_{0}u^{2} dx \\
&- \frac{t_{2}^{2}}{t_{1}^{2}- t_{2}^{2}} \int_{(\R^{3}\setminus \Omega_{\e})\cap \{t_{2}u\leq a <t_{1}u\}} \frac{V_{0}}{K}u^{2} dx + \frac{t_{1}^{2}}{t_{1}^{2}- t_{2}^{2}} \int_{(\R^{3}\setminus \Omega_{\e})\cap \{t_{2}u\leq a <t_{1}u\}} \frac{f(t_{2}u)}{t_{2}u} u^{2} dx \\
&\leq \frac{1}{K} \int_{\R^{3}\setminus \Omega_{\e}} V_{0}u^{2} dx \leq \frac{1}{K} \|u\|_{\e}^{2}. 
\end{align*}
Then, by using the fact $u\neq 0$ and $K>\frac{2}{m_{0}}$, we get $m_{0}\leq \frac{1}{K}<m_{0}$, and this is a contradiction. 

\noindent
$(ii)$ Let $u\in \mathbb{S}_{\e}^{+}$. By $(i)$ there exists $t_{u}>0$ such that $h_{u}'(t_{u})=0$, or equivalently
\begin{equation*}
t_{u} M(t_{u}^{2})= \int_{\R^{3}} g(\e x, t_{u}u) \,u \,dx. 
\end{equation*}
By assumptions $(g_{1})$ and $(g_{2})$, given $\xi>0$ there exists a positive constant $C_{\xi}$ such that
\begin{equation*}
|G(x, t)|\leq \xi t^{4} + C_{\xi} |t|^{q}, \quad \mbox{ for every } t\in \R.
\end{equation*} 
The above inequality, the assumption $(M_1)$ and the Sobolev embedding in Theorem \ref{Sembedding} yield
\begin{align*}
m_{0}t_{u}&\leq M(t_{u}^{2}) t_{u}= \int_{\R^{3}} g(\e x, t_{u}u)\, u\, dx\\
&\leq \xi \frac{t_{u}^{4}}{4}C_{1} + C_{\xi} \frac{t_{u}^{q}}{q}C_{2}. 
\end{align*}
Thus we obtain that there exists $\tau>0$, independent of $u$, such that $t_{u}\geq \tau$. Now, let $\mathbb{K}\subset \mathbb{S}_{\e}^{+}$ be a compact set and we show that $t_{u}$ can be estimated from above by a constant depending on $\mathbb{K}$. Assume by contradiction that there exists a sequence $\{u_{n}\}_{n\in \mathbb{N}}\subset \mathbb{K}$ such that $t_{n}:=t_{u_{n}}\rightarrow \infty$. Therefore, there exists $u\in \mathbb{K}$ such that $u_{n}\rightarrow u$ in $\mathcal{H}_{\e}$. From \eqref{b2.2} we get 
\begin{equation}\label{canc}
\mathcal{J}_{\e}(t_{n}u_{n})\rightarrow -\infty. 
\end{equation}
Fix $v\in \N_{\e}$. Then, by using the fact that $\langle \mathcal{J}_{\e}'(v), v \rangle=0$, and the assumptions $(g_{3})$-(i) and $(g_{3})$-(ii), we can infer
\begin{align*}
\mathcal{J}_{\e}(v)&= \mathcal{J}_{\e}(v)- \frac{1}{\vartheta} \langle \mathcal{J}_{\e}'(v), v \rangle\\
&=\frac{1}{2}\widehat{M}(\|v\|_{\e}^{2})-\frac{1}{\vartheta} M(\|v\|_{\e}^{2})\|v\|_{\e}^{2}+ \frac{1}{\vartheta}\int_{\R^{3}} [g(\e x, v)v- \vartheta G(\e x, v)]\, dx\\
&=\frac{1}{2}\widehat{M}(\|v\|_{\e}^{2})-\frac{1}{\vartheta} M(\|v\|_{\e}^{2})\|v\|_{\e}^{2}+ \frac{1}{\vartheta}\int_{\R^{3}\setminus \Omega_{\e}} [g(\e x, v)v- \vartheta G(\e x, v)]\, dx + \frac{1}{\vartheta}\int_{\Omega_{\e}} [g(\e x, v)v- \vartheta G(\e x, v)]\, dx\\
&\geq \frac{1}{2}\widehat{M}(\|v\|_{\e}^{2})-\frac{1}{\vartheta} M(\|v\|_{\e}^{2})\|v\|_{\e}^{2}+ \frac{1}{\vartheta}\int_{\R^{3}\setminus \Omega_{\e}} [g(\e x, v)v- \vartheta G(\e x, v)]\, dx \\
&\geq \frac{1}{2}\widehat{M}(\|v\|_{\e}^{2})-\frac{1}{\vartheta} M(\|v\|_{\e}^{2})\|v\|_{\e}^{2} -\left(\frac{\vartheta-2}{2\vartheta}\right) \frac{1}{K}\int_{\R^{3}\setminus \Omega_{\e}} V(\e x) v^{2}dx\\
& \geq \frac{1}{2}\widehat{M}(\|v\|_{\e}^{2})-\frac{1}{\vartheta} M(\|v\|_{\e}^{2})\|v\|_{\e}^{2}-\left(\frac{\vartheta-2}{2\vartheta}\right)\frac{1}{K}\|v\|_{\e}^{2}.
\end{align*}
Now, by using $(M_3)$, we know that 
\begin{equation}\label{ccc} 
\widehat{M}(t)\geq \frac{M(t)+m_{0}}{2} t \, \mbox{ for any } t\geq 0.
\end{equation} 
This together with $(M_{1})$ implies that
\begin{align}\label{cancan}
\mathcal{J}_{\e}(v) &\geq \frac{1}{4} \left[M(\|v\|_{\e}^{2}) + m_{0} \right]\|v\|_{\e}^{2} - \frac{1}{\vartheta} M(\|v\|_{\e}^{2}) \|v\|_{\e}^{2} - \left(\frac{\vartheta-2}{2\vartheta}\right)\frac{1}{K}\|v\|_{\e}^{2}\nonumber\\
&= \frac{\vartheta-4}{4\vartheta} M(\|v\|_{\e}^{2}) \|v\|_{\e}^{2} + \frac{1}{4}m_{0} \|v\|_{\e}^{2} - \left(\frac{\vartheta-2}{2\vartheta}\right)\frac{1}{K}\|v\|_{\e}^{2}\nonumber \\
&\geq \left(\frac{\vartheta-4}{4\vartheta} + \frac{1}{4}\right) m_{0}\|v\|_{\e}^{2} - \left(\frac{\vartheta-2}{2\vartheta}\right) \frac{1}{K}\|v\|_{\e}^{2} \nonumber \\
&= \left(\frac{\vartheta-2}{2\vartheta}\right) \left( m_{0}- \frac{1}{K}\right) \|v\|_{\e}^{2}.
\end{align}
Taking into account that $\{t_{u_{n}}u_{n}\}_{n\in \mathbb{N}}\subset \N_{\e}$ and $K>\frac{2}{m_{0}}$, from \eqref{cancan} we deduce that \eqref{canc} does not hold. 

\noindent
$(iii)$ Firstly, we note that $\hat{m}_{\e}$, $m_{\e}$ and $m_{\e}^{-1}$ are well defined. Indeed, by $(i)$ for each $u\in \mathcal{H}_{\e}^{+}$ there exists a unique $m_{\e}(u)\in \N_{\e}$. On the other hand, if $u\in \N_{\e}$ then $u\in \mathcal{H}_{\e}^{+}$. Otherwise if $u\notin \mathcal{H}_{\e}^{+}$, we have
\begin{equation*}
|\supp (u^{+}) \cap \Omega_{\e}|=0, 
\end{equation*}
which together with $(g3)$-(ii) gives 
\begin{align}\label{ter1}
0<M(\|u\|_{\e}^{2})\|u\|_{\e}^{2}&= \int_{\R^{3}} g(\e x, u)\, u \, dx \nonumber \\
&= \int_{\R^{3}\setminus \Omega_{\e}} g(\e x, u)\, u \, dx + \int_{\Omega_{\e}} g(\e x, u)\, u \, dx \nonumber \\
&= \int_{\R^{3}\setminus \Omega_{\e}} g(\e x, u^{+})\, u^{+} \, dx \nonumber \\
&\leq \frac{1}{K} \int_{\R^{3}\setminus \Omega_{\e}} V(\e x) u^{2} dx \leq \frac{1}{K} \|u\|_{\e}^{2}.  
\end{align}
By using $(M_{1})$ and \eqref{ter1} we get
\begin{align*}
0<m_{0} \|u\|_{\e}^{2}\leq M(\|u\|_{\e}^{2}) \|u\|_{\e}^{2} \leq \frac{1}{K} \|u\|_{\e}^{2}
\end{align*}
and this leads to a contradiction because $\frac{1}{K}<\frac{m_{0}}{2}$. 
As a consequence $m_{\e}^{-1}(u)= \frac{u}{\|u\|_{\e}}\in \mathbb{S}_{\e}^{+}$, $m_{\e}^{-1}$ is well defined and continuous. \\
Now, let $u\in \mathbb{S}_{\e}^{+}$, then
\begin{align*}
m_{\e}^{-1}(m_{\e}(u))= m_{\e}^{-1}(t_{u}u)= \frac{t_{u}u}{\|t_{u}u\|_{\e}}= \frac{u}{\|u\|_{\e}}=u
\end{align*}
from which $m_{\e}$ is a bijection. Now, our aim is to prove that $\hat{m}_{\e}$ is a continuous function. Let $\{u_{n}\}_{n\in \mathbb{N}}\subset \mathcal{H}_{\e}^{+}$ and $u\in \mathcal{H}_{\e}^{+}$ such that $u_{n}\rightarrow u$ in $\mathcal{H}_{\e}^{+}$. Thus,
\begin{equation*}
\frac{u_{n}}{\|u_{n}\|_{\e}}\rightarrow \frac{u}{\|u\|_{\e}} \mbox{ in } \mathcal{H}_{\e}.
\end{equation*} 
Let $v_{n}:= \frac{u_{n}}{\|u_{n}\|_{\e}}$ and $t_{n}:=t_{v_{n}}$. By $(ii)$ there exists $t_{0}>0$ such that $t_{n}\rightarrow t_{0}$. Since $t_{n}v_{n}\in \N_{\e}$ and $\|v_{n}\|_{\e}=1$, we have 
\begin{equation*}
M(t_{n}^{2})\,t_{n}= \int_{\R^{3}} g(\e x, t_{n}v_{n})\, v_{n} \, dx, 
\end{equation*}
or equivalently
\begin{equation*}
M(t_{n}^{2})\,t_{n}= \frac{1}{\|u_{n}\|_{\e}}\int_{\R^{3}} g(\e x, t_{n}v_{n}) \,u_{n} \, dx. 
\end{equation*}
By passing to the limit as $n\rightarrow \infty$ we obtain
\begin{equation*}
M(t_{0}^{2}) \,t_{0}= \frac{1}{\|u\|_{\e}} \int_{\R^{3}} g(\e x, t_{0}v) \, u\, dx, 
\end{equation*}
where $v= \frac{u}{\|u\|_{\e}}$, that implies that $t_{0}v\in \N_{\e}$. By $(i)$ we deduce that $t_{v}= t_{0}$, and this shows that
\begin{equation*}
\hat{m}_{\e}(u_{n})= \hat{m}_{\e}\left(\frac{u_{n}}{\|u_{n}\|_{\e}}\right)\rightarrow \hat{m}_{\e}\left(\frac{u}{\|u\|_{\e}}\right)=\hat{m}_{\e}(u) \, \mbox{ in } \mathcal{H}_{\e}.  
\end{equation*}
Therefore $\hat{m}_{\e}$ and $m_{\e}$ are continuous functions. 

\noindent
$(iv)$ Let $\{u_{n}\}_{n\in \mathbb{N}}\subset \mathbb{S}_{\e}^{+}$ be such that ${\rm dist}(u_{n}, \partial \mathbb{S}_{\e}^{+})\rightarrow 0$. Observing that for each $p\in [2, 2^{*}_{s}]$ and $n\in \mathbb{N}$ it holds
\begin{align*}
\|u_{n}^{+}\|_{L^{p}(\Omega_{\e})} &\leq \inf_{v\in \partial \mathcal{S}_{\e}^{+}} \|u_{n}- v\|_{L^{p}(\Omega_{\e})}\\
&\leq C_{p} \inf_{v\in \partial \mathcal{S}_{\e}^{+}} \|u_{n}- v\|_{\e},
\end{align*}
by $(g_{1})$, $(g_{2})$, and $(g_{3})$-(ii), we can infer
\begin{align*}
\int_{\R^{3}} G(\e x, tu_{n})\, dx &= \int_{\R^{3}\setminus \Omega_{\e}} G(\e x, tu_{n})\, dx + \int_{\Omega_{\e}} G(\e x, tu_{n})\, dx\\
&\leq \frac{t^{2}}{K} \int_{\R^{3}\setminus \Omega_{\e}} V(\e x) u_{n}^{2} dx+ \int_{\Omega_{\e}} F(tu_{n})\, dx\\
&\leq \frac{t^{2}}{K} \|u_{n}\|_{\e}^{2} + C_{1}t^{4} \int_{\Omega_{\e}} (u_{n}^{+})^{4} dx+ C_{2} t^{q} \int_{\Omega_{\e}} (u_{n}^{+})^{q} dx\\
&\leq \frac{t^{2}}{K}+ C_{1}' t^{4} {\rm dist}(u_{n}, \partial \mathcal{S}_{\e}^{+})^{4} + C_{2}' {\rm dist}(u_{n}, \partial \mathcal{S}_{\e}^{+})^{q}
\end{align*}
from which, for all $t>0$ 
\begin{equation}\label{ter2}
\limsup_{n\rightarrow \infty} \int_{\R^{3}} G(\e x, tu_{n})\, dx \leq \frac{t^{2}}{K}. 
\end{equation}
Taking in mind the definitions of $m_{\e}(u_{n})$ and $\widehat{M}(t)$, and by using \eqref{ter2} and assumption $(M_{1})$ we have
\begin{align*}
\liminf_{n\rightarrow \infty} \mathcal{J}_{\e}(m_{\e}(u_{n}))&\geq \liminf_{n\rightarrow \infty} \mathcal{J}_{\e}(m_{\e}(u_{n}))\\
&= \liminf_{n\rightarrow \infty} \left[ \frac{1}{2} \widehat{M}(\|tu_{n}\|_{\e}^{2}) - \int_{\R^{3}} G(\e x, tu_{n})\, dx\right]\\
&\geq \frac{1}{2} \widehat{M}(t^{2}) -\frac{t^{2}}{K}\\
&\geq \left(\frac{1}{2} m_{0}- \frac{1}{K} \right)t^{2}.
\end{align*}
Recalling that $K>2/m_{0}$ we get
\begin{equation*}
\lim_{n\rightarrow \infty} \mathcal{J}_{\e}(m_{\e}(u_{n}))= \infty. 
\end{equation*}
Moreover, the definition of $\mathcal{J}_{\e}(m_{\e}(u_{n}))$ and \eqref{cc} yield that
\begin{equation*}
\mathcal{J}_{\e}(m_{\e}(u_{n})) \leq \frac{1}{2}\widehat{M}(t^{2}_{u_{n}}) \leq C (t_{u_{n}} + t^{2}_{u_{n}})\quad \mbox{ for each } n\in \mathbb{N}. 
\end{equation*}
which gives $\|m_{\e}(u_{n})\|_{\e}\rightarrow \infty$ as $n\rightarrow \infty$. 
\end{proof}

\noindent
Let us define the maps 
\begin{equation*}
\hat{\psi}_{\e}: \mathcal{H}_{\e}^{+} \rightarrow \R \quad \mbox{ and } \quad \psi_{\e}: \mathbb{S}_{\e}^{+}\rightarrow \R, 
\end{equation*}
by $\hat{\psi}_{\e}(u):= \mathcal{J}_{\e}(\hat{m}_{\e}(u))$ and $\psi_{\e}:=\hat{\psi}_{\e}|_{\mathbb{S}_{\e}^{+}}$. \\
The next result is a consequence of Lemma \ref{lem2.3}. 
\begin{prop}\label{prop2.1}
Assume that the hypotheses $(M_{1})$-$(M_{3})$, $(V_{1})$-$(V_{2})$ and $(f_{1})$-$(f_{4})$ hold true. Then, 
\begin{compactenum}[$(a)$]
\item $\hat{\psi}_{\e} \in C^{1}(\mathcal{H}_{\e}^{+}, \R)$ and 
\begin{equation*}
\langle \hat{\psi}_{\e}'(u), v\rangle = \frac{\|\hat{m}_{\e}(u)\|_{\e}}{\|u\|_{\e}} \langle \mathcal{J}_{\e}'(\hat{m}_{\e}(u)), v\rangle 
\end{equation*}
for every $u\in \mathcal{H}_{\e}^{+}$ and $v\in \mathcal{H}_{\e}$; 
\item $\psi_{\e} \in C^{1}(\mathbb{S}_{\e}^{+}, \R)$ and 
\begin{equation*}
\langle \psi_{\e}'(u), v \rangle = \|m_{\e}(u)\|_{\e} \langle \mathcal{J}_{\e}'(m_{\e}(u)), v\rangle, 
\end{equation*}
for every 
\begin{equation*}
v\in T_{u}\mathbb{S}_{\e}^{+}:= \left \{ v\in \mathcal{H}_{\e} : (u, v)_{\e}=0 \right\}; 
\end{equation*}
\item If $\{u_{n}\}_{n\in \mathbb{N}}$ is a $(PS)_{d}$ sequence for $\psi_{\e}$, then $\{m_{\e}(u_{n})\}_{n\in \mathbb{N}}$ is a $(PS)_{d}$ sequence for $\mathcal{J}_{\e}$. If $\{u_{n}\}_{n\in \mathbb{N}}\subset \N_{\e}$ is a bounded $(PS)_{d}$ sequence for $\mathcal{J}_{\e}$, then $\{m_{\e}^{-1}(u_{n})\}_{n\in \mathbb{N}}$ is a $(PS)_{d}$ sequence for the functional $\psi_{\e}$; 
\item $u$ is a critical point of $\psi_{\e}$ if, and only if, $m_{\e}(u)$ is a nontrivial critical point for $\mathcal{J}_{\e}$. Moreover, the corresponding critical values coincide and 
\begin{equation*}
\inf_{u\in \mathbb{S}_{\e}^{+}} \psi_{\e}(u)= \inf_{u\in \N_{\e}} \mathcal{J}_{\e}(u).  
\end{equation*}
\end{compactenum}
\end{prop}

\begin{remark}\label{rem3}
As in \cite{SW}, we can see that, thanks to the assumptions $(M_1)$-$(M_3)$, the following equalities hold
\begin{align*}
c_{\e}&:=\inf_{u\in \N_{\e}} \mathcal{J}_{\e}(u)=\inf_{u\in \mathcal{H}_{\e}^{+}} \max_{t>0} \mathcal{J}_{\e}(tu)=\inf_{u\in \mathbb{S}_{\e}^{+}} \max_{t>0} \mathcal{J}_{\e}(tu).
\end{align*}
\end{remark}

\noindent
Now, we aim to show that the functional $\mathcal{J}_{\e}$ satisfies the Palais-Smale condition.
Firstly, we prove the following.
\begin{lem}\label{lem2.4}
Let $\{u_{n}\}_{n\in\mathbb{N}}$ be a $(PS)_{d}$ sequence for $\mathcal{J}_{\e}$. Then $\{u_{n}\}_{n\in\mathbb{N}}$ is bounded in $\mathcal{H}_{\e}$.
\end{lem}

\begin{proof}
Let $\{u_{n}\}_{n\in\mathbb{N}}$ be a $(PS)$ sequence at the level $d$, that is 
\begin{equation*}
\mathcal{J}_{\e}(u_{n})\rightarrow d \, \mbox{ and } \, \mathcal{J}_{\e}'(u_{n})\rightarrow 0 \mbox{ in } \mathcal{H}_{\e}^{-1}.
\end{equation*}
Then, arguing as in the proof of Lemma \ref{lem2.3}-$(ii)$ (see formula \eqref{cancan} there), we can see that 
\begin{align*}
C+ \|u_{n}\|_{\e}&\geq \mathcal{J}_{\e}(u_{n})- \frac{1}{\vartheta} \langle \mathcal{J}_{\e}'(u_{n}), u_{n}\rangle\\
&\geq \left( \frac{\vartheta -2}{2\vartheta}\right) \left(m_{0}- \frac{1}{K}\right) \|u_{n}\|_{\e}^{2}. 
\end{align*}
By using the fact that $\vartheta >4$ and $K>2/m_{0}$, we deduce that $\{u_{n}\}_{n\in\mathbb{N}}$ is bounded in $\mathcal{H}_{\e}$.
\end{proof}

\noindent
The next two lemmas are fundamental to obtain compactness of bounded Palais-Smale sequences.
\begin{lem}\label{lem2.5}
Let $\{u_{n}\}_{n\in\mathbb{N}}$ be a $(PS)_{d}$ sequence for $\mathcal{J}_{\e}$. Then, for each $\zeta>0$, there exists $R=R(\zeta)>0$ such that 
\begin{equation*}
\limsup_{n\rightarrow \infty} \left[ \int_{\R^{3}\setminus B_{R}} \,dx \int_{\R^{3}} \frac{|u_{n}(x)- u_{n}(y)|^{2}}{|x-y|^{3+2s}}\,dy + \int_{\R^{3}\setminus B_{R}} V(\e x) u_{n}^{2} \,dx\right] <\zeta. 
\end{equation*}
\end{lem}

\begin{proof}
For any $R>0$, let $\eta_{R}\in C^{\infty}(\R^{3})$ be such that $\eta_{R}=0$ in $B_{R}$ and $\eta_{R}=1$ in $B_{2R}^{c}$, with $0\leq \eta_{R}\leq 1$ and $|\nabla \eta_{R}|\leq \frac{C}{R}$, where $C$ is a constant independent of $R$. 
Since $\{\eta_{R}u_{n}\}_{n\in \mathbb{N}}$ is bounded in $\mathcal{H}_{\e}$, it follows that $\langle \mathcal{J}_{\e}'(u_{n}), \eta_{R}u_{n}\rangle =o_{n}(1)$, that is
\begin{align*}
&M(\|u_{n}\|_{\e}^{2}) \left[\iint_{\R^{6}} \frac{|u_{n}(x)- u_{n}(y)|^{2}}{|x-y|^{3+2s}} \eta_{R}(x) \,dxdy + \int_{\R^{3}} V(\e x) u_{n}^{2} \eta_{R} \, dx\right] \\
&=o_{n}(1) + \int_{\R^{3}} g(\e x, u_{n}) u_{n}\eta_{R} \, dx - M(\|u_{n}\|_{\e}^{2})\iint_{\R^{6}} \frac{(\eta_{R}(x)- \eta_{R}(y))(u_{n}(x)- u_{n}(y))}{|x-y|^{3+2s}} u_{n}(y) \,dxdy. 
\end{align*}
Take $R>0$ such that $\Omega_{\e}\subset B_{R}$. Then, by using $(M_{1})$ and $(g_{3})$-(ii) we have
\begin{align*}
&m_{0}\left[\iint_{\R^{6}} \frac{|u_{n}(x)- u_{n}(y)|^{2}}{|x-y|^{3+2s}} \eta_{R}(x) \,dxdy + \int_{\R^{3}} V(\e x) u_{n}^{2} \eta_{R} \, dx\right] \\
&\leq \int_{\R^{3}} \frac{1}{K} V(\e x) u_{n}^{2} \eta_{R}\, dx - M(\|u_{n}\|_{\e}^{2})\iint_{\R^{6}} \frac{(\eta_{R}(x)- \eta_{R}(y))(u_{n}(x)- u_{n}(y))}{|x-y|^{3+2s}} u_{n}(y) \,dxdy + o_{n}(1)
\end{align*}
which gives 
\begin{align}\label{ter3}
\left(m_{0}- \frac{1}{K}\right)& \left[ \iint_{\R^{6}} \frac{|u_{n}(x)- u_{n}(y)|^{2}}{|x-y|^{3+2s}} \eta_{R}(x) \,dxdy + \int_{\R^{3}} V(\e x) u_{n}^{2} \eta_{R} \, dx\right] \nonumber \\
&\leq - M(\|u_{n}\|_{\e}^{2})\iint_{\R^{6}} \frac{(\eta_{R}(x)- \eta_{R}(y))(u_{n}(x)- u_{n}(y))}{|x-y|^{3+2s}} u_{n}(y) \,dxdy + o_{n}(1).
\end{align}
Let us observe that the boundedness of $\{u_{n}\}_{n\in \mathbb{N}}$ in $\mathcal{H}_{\e}$ and the assumption $(M_{2})$, imply
\begin{equation}\label{ter4}
M(\|u_{n}\|_{\e}^{2})\leq C \mbox{ for any } n\in \mathbb{N}.
\end{equation}
Now we claim that 
\begin{equation}\label{ter5}
\lim_{R\rightarrow \infty} \limsup_{n\rightarrow \infty} \iint_{\R^{6}} \frac{(\eta_{R}(x)- \eta_{R}(y))(u_{n}(x)- u_{n}(y))}{|x-y|^{3+2s}} u_{n}(y) \,dxdy=0.
\end{equation}
Since
\begin{align*}
&\left|\iint_{\R^{6}} \frac{(\eta_{R}(x)- \eta_{R}(y))(u_{n}(x)- u_{n}(y))}{|x-y|^{3+2s}} u_{n}(y) \,dxdy\right|\nonumber \\
&\leq \left( \iint_{\R^{6}} \frac{|u_{n}(x)- u_{n}(y)|^{2}}{|x-y|^{3+2s}}\,dxdy\right)^{\frac{1}{2}}\left( \iint_{\R^{6}} \frac{|\eta_{R}(x)- \eta_{R}(y)|^{2}}{|x-y|^{3+2s}} u_{n}^{2}(y)\,dxdy\right)^{\frac{1}{2}}\nonumber\\
&\leq C \left( \iint_{\R^{6}} \frac{|\eta_{R}(x)- \eta_{R}(y)|^{2}}{|x-y|^{3+2s}} u_{n}^{2}(y)\,dxdy\right)^{\frac{1}{2}}
\end{align*}
it is enough to show 
\begin{equation*}
\lim_{R\rightarrow \infty} \limsup_{n\rightarrow \infty} \iint_{\R^{6}} \frac{|\eta_{R}(x)- \eta_{R}(y)|^{2}}{|x-y|^{3+2s}} u_{n}^{2}(y)\,dxdy=0
\end{equation*}
to deduce that \eqref{ter5} holds. \\
Firstly, we note that $\R^{6}$ can be written as 
$$
\R^{6}=((\R^{3}\setminus B_{2R})\times (\R^{3}\setminus B_{2R})) \cup ((\R^{3}\setminus B_{2R})\times B_{2R})\cup (B_{2R}\times \R^{3})=: \mathbb{X}^{1}_{R}\cup \mathbb{X}^{2}_{R} \cup \mathbb{X}^{3}_{R}.
$$
Then
\begin{align}\label{Pa1}
\iint_{\R^{6}} &\frac{|\eta_{R}(x)-\eta_{R}(y)|^{2}}{|x-y|^{3+2s}} u^{2}_{n}(x) dx dy =\iint_{\mathbb{X}^{1}_{R}}\frac{|\eta_{R}(x)-\eta_{R}(y)|^{2}}{|x-y|^{3+2s}} u^{2}_{n}(x) dx dy \nonumber \\
&+\iint_{\mathbb{X}^{2}_{R}}\frac{|\eta_{R}(x)-\eta_{R}(y)|^{2}}{|x-y|^{3+2s}} u^{2}_{n}(x) dx dy+
\iint_{\mathbb{X}^{3}_{R}}\frac{|\eta_{R}(x)-\eta_{R}(y)|^{2}}{|x-y|^{3+2s}} u^{2}_{n}(x) dx dy.
\end{align}
Now, we estimate each integrals in (\ref{Pa1}).
Since $\eta_{R}=1$ in $\R^{3}\setminus B_{2R}$, we have
\begin{align}\label{Pa2}
\iint_{\mathbb{X}^{1}_{R}}\frac{|u_{n}(x)|^{2}|\eta_{R}(x)-\eta_{R}(y)|^{2}}{|x-y|^{3+2s}} dx dy=0.
\end{align}
Let $k>4$. Clearly, we have
\begin{equation*}
\mathbb{X}^{2}_{R}=(\R^{3} \setminus B_{2R})\times B_{2R} = ((\R^{6}\setminus B_{kR})\times B_{2R})\cup ((B_{kR}\setminus B_{2R})\times B_{2R}). 
\end{equation*}
Let us observe that, if $(x, y) \in (\R^{6}\setminus B_{kR})\times B_{2R}$, then
\begin{equation*}
|x-y|\geq |x|-|y|\geq |x|-2R>\frac{|x|}{2}. 
\end{equation*}
Therefore, taking into account that $0\leq \eta_{R}\leq 1$, $|\nabla \eta_{R}|\leq \frac{C}{R}$ and applying H\"older inequality, we can see
\begin{align}\label{Pa3}
&\iint_{\mathbb{X}^{2}_{R}}\frac{|u_{n}(x)|^{2}|\eta_{R}(x)-\eta_{R}(y)|^{2}}{|x-y|^{3+2s}} dx dy \nonumber \\
&=\int_{\R^{3}\setminus B_{kR}} dx \int_{B_{2R}} \frac{|u_{n}(x)|^{2}|\eta_{R}(x)-\eta_{R}(y)|^{2}}{|x-y|^{3+2s}} dy + \int_{B_{kR}\setminus B_{2R}} dx  \int_{B_{2R}} \frac{|u_{n}(x)|^{2}|\eta_{R}(x)-\eta_{R}(y)|^{2}}{|x-y|^{3+2s}} dy \nonumber \\
&\leq 2^{2+3+2s} \int_{\R^{3}\setminus B_{kR}} dx  \int_{B_{2R}} \frac{|u_{n}(x)|^{2}}{|x|^{3+2s}}\, dy+ \frac{C}{R^{2}} \int_{B_{kR}\setminus B_{2R}} dx  \int_{B_{2R}} \frac{|u_{n}(x)|^{2}}{|x-y|^{3+2(s-1)}}\, dy \nonumber \\
&\leq CR^{3} \int_{\R^{3}\setminus B_{kR}} \frac{|u_{n}(x)|^{2}}{|x|^{3+2s}}\, dx + \frac{C}{R^{2}} (kR)^{2(1-s)} \int_{B_{kR}\setminus B_{2R}} |u_{n}(x)|^{2} dx \nonumber \\
&\leq CR^{3} \left( \int_{\R^{3}\setminus B_{kR}} |u_{n}(x)|^{2^{*}_{s}} dx \right)^{\frac{2}{2^{*}_{s}}} \left(\int_{\R^{3}\setminus B_{kR}}\frac{1}{|x|^{\frac{3^{2}}{2s} +3}}\, dx \right)^{\frac{2s}{3}} + \frac{C k^{2(1-s)}}{R^{2s}} \int_{B_{kR}\setminus B_{2R}} |u_{n}(x)|^{2} dx \nonumber \\
&\leq \frac{C}{k^{3}} \left( \int_{\R^{3}\setminus B_{kR}} |u_{n}(x)|^{2^{*}_{s}} dx \right)^{\frac{2}{2^{*}_{s}}} + \frac{C k^{2(1-s)}}{R^{2s}} \int_{B_{kR}\setminus B_{2R}} |u_{n}(x)|^{2} dx \nonumber \\
&\leq \frac{C}{k^{3}}+ \frac{C k^{2(1-s)}}{R^{2s}} \int_{B_{kR}\setminus B_{2R}} |u_{n}(x)|^{2} dx.
\end{align}

\noindent
Now, we fix $\delta \in (0,1)$, and we note that
\begin{align}\label{TerV1}
&\iint_{\mathbb{X}^{3}_{R}} \frac{|u_{n}(x)|^{2} |\eta_{R}(x)- \eta_{R}(y)|^{2}}{|x-y|^{3+2s}}\, dxdy \nonumber\\
&= \int_{B_{2R}\setminus B_{\delta R}}dx  \int_{\R^{3}} \frac{|u_{n}(x)|^{2} |\eta_{R}(x)- \eta_{R}(y)|^{2}}{|x-y|^{3+2s}}\, dy + \int_{B_{\delta R}} dx \int_{\R^{3}} \frac{|u_{n}(x)|^{2} |\eta_{R}(x)- \eta_{R}(y)|^{2}}{|x-y|^{3+2s}}\, dy. 
\end{align}
Let us estimate the first integral in \eqref{TerV1}. Then, 
\begin{align*}
\int_{B_{2R}\setminus B_{\delta R}} dx \int_{\R^{3} \cap \{y: |x-y|<R\}} \frac{|u_{n}(x)|^{2} |\eta_{R}(x)- \eta_{R}(y)|^{2}}{|x-y|^{3+2s}}\, dy \leq \frac{C}{R^{2s}} \int_{B_{2R}\setminus B_{\delta R}} |u_{n}(x)|^{2} dx
\end{align*}
and 
\begin{align*}
\int_{B_{2R}\setminus B_{\delta R}} dx \int_{\R^{3} \cap \{y: |x-y|\geq R\}} \frac{|u_{n}(x)|^{2} |\eta_{R}(x)- \eta_{R}(y)|^{2}}{|x-y|^{3+2s}}\, dy \leq \frac{C}{R^{2s}} \int_{B_{2R}\setminus B_{\delta R}} |u_{n}(x)|^{2} dx
\end{align*}
from which we have
\begin{align}\label{TerV2}
\int_{B_{2R}\setminus B_{\delta R}} dx \int_{\R^{3}} \frac{|u_{n}(x)|^{2} |\eta_{R}(x)- \eta_{R}(y)|^{2}}{|x-y|^{3+2s}}\, dy \leq \frac{C}{R^{2s}} \int_{B_{2R}\setminus B_{\delta R}} |u_{n}(x)|^{2} dx. 
\end{align}
Now, by using the definition of $\eta_{R}$, $\delta \in (0,1)$, and $0\leq\eta_{R}\leq 1$, we have 
\begin{align}\label{TerV3}
\int_{B_{\delta R}} dx \int_{\R^{3}} \frac{|u_{n}(x)|^{2} |\eta_{R}(x)- \eta_{R}(y)|^{2}}{|x-y|^{3+2s}}\, dy &= \int_{B_{\delta R}} dx \int_{\R^{3}\setminus B_{R}} \frac{|u_{n}(x)|^{2} |\eta_{R}(x)- \eta_{R}(y)|^{2}}{|x-y|^{3+2s}}\, dy\nonumber \\
&\leq 4 \int_{B_{\delta R}} dx \int_{\R^{3}\setminus B_{R}} \frac{|u_{n}(x)|^{2}}{|x-y|^{3+2s}}\, dy\nonumber \\
&\leq C \int_{B_{\delta R}} |u_{n}|^{2} dx \int_{(1-\delta)R}^{\infty} \frac{1}{r^{1+2s}} dr\nonumber \\
&=\frac{C}{[(1-\delta)R]^{2s}} \int_{B_{\delta R}} |u_{n}|^{2} dx
\end{align}
where we have used the fact that if $(x, y) \in B_{\delta R}\times (\R^{3} \setminus B_{R})$, then $|x-y|>(1-\delta)R$. \\
Taking into account \eqref{TerV1}, \eqref{TerV2} and \eqref{TerV3} we deduce 
\begin{align}\label{Pa4}
\iint_{\mathbb{X}^{3}_{R}} &\frac{|u_{n}(x)|^{2} |\eta_{R}(x)- \eta_{R}(y)|^{2}}{|x-y|^{3+2s}}\, dxdy \nonumber \\
&\leq \frac{C}{R^{2s}} \int_{B_{2R}\setminus B_{\delta R}} |u_{n}(x)|^{p} dx + \frac{C}{[(1-\delta)R]^{2s}} \int_{B_{\delta R}} |u_{n}(x)|^{2} dx. 
\end{align}
Putting together \eqref{Pa1}, \eqref{Pa2}, \eqref{Pa3} and \eqref{Pa4}, we can infer 
\begin{align}\label{Pa5}
\iint_{\R^{6}} \frac{|u_{n}(x)|^{2} |\eta_{R}(x)- \eta_{R}(y)|^{2}}{|x-y|^{3+2s}}\, dxdy &\leq \frac{C}{k^{3}} + \frac{Ck^{2(1-s)}}{R^{2s}} \int_{B_{kR}\setminus B_{2R}} |u_{n}(x)|^{2} dx \nonumber\\
&+ \frac{C}{R^{2s}} \int_{B_{2R}\setminus B_{\delta R}} |u_{n}(x)|^{2} dx + \frac{C}{[(1-\delta)R]^{2s}}\int_{B_{\delta R}} |u_{n}(x)|^{2} dx. 
\end{align}
Since $\{u_{n}\}_{n\in \mathbb{N}}$ is bounded in $H^{s}(\R^{3})$, by using Theorem \ref{Sembedding}, we may assume that $u_{n}\rightarrow u$ in $L^{2}_{loc}(\R^{3})$ for some $u\in H^{s}(\R^{3})$. Then, taking the limit as $n\rightarrow \infty$ in \eqref{Pa5}, we have
\begin{align*}
&\limsup_{n\rightarrow \infty} \iint_{\R^{6}} \frac{|u_{n}(x)|^{2} |\eta_{R}(x)- \eta_{R}(y)|^{2}}{|x-y|^{3+2s}}\, dxdy\\
&\leq \frac{C}{k^{3}} + \frac{Ck^{2(1-s)}}{R^{2s}} \int_{B_{kR}\setminus B_{2R}} |u(x)|^{2} dx + \frac{C}{R^{2s}} \int_{B_{2R}\setminus B_{\delta R}} |u(x)|^{2} dx + \frac{C}{[(1-\delta)R]^{2s}}\int_{B_{\delta R}} |u(x)|^{2} dx \\
&\leq \frac{C}{k^{3}} + Ck^{2} \left( \int_{B_{kR}\setminus B_{2R}} |u(x)|^{2^{*}_{s}} dx\right)^{\frac{2}{2^{*}_{s}}} + C\left(\int_{B_{2R}\setminus B_{\delta R}} |u(x)|^{2^{*}_{s}} dx\right)^{\frac{2}{2^{*}_{s}}} + C\left( \frac{\delta}{1-\delta}\right)^{2s} \left(\int_{B_{\delta R}} |u(x)|^{2^{*}_{s}} dx\right)^{\frac{2}{2^{*}_{s}}}, 
\end{align*}
where in the last passage we have used the H\"older inequality. \\
Since $u\in L^{2^{*}_{s}}(\R^{3})$, $k>4$ and $\delta \in (0,1)$, we obtain
\begin{align*}
\limsup_{R\rightarrow \infty} \int_{B_{kR}\setminus B_{2R}} |u(x)|^{2^{*}_{s}} dx = \limsup_{R\rightarrow \infty} \int_{B_{2R}\setminus B_{\delta R}} |u(x)|^{2^{*}_{s}} dx = 0. 
\end{align*}
Choosing $\delta= \frac{1}{k}$, we get
\begin{align*}
&\limsup_{R\rightarrow \infty} \limsup_{n\rightarrow \infty} \iint_{\R^{6}} \frac{|u_{n}(x)|^{2} |\eta_{R}(x)- \eta_{R}(y)|^{2}}{|x-y|^{3+2s}}\, dxdy\\
&\quad \leq \lim_{k\rightarrow \infty} \limsup_{R\rightarrow \infty} \Bigl[\, \frac{C}{k^{3}} + Ck^{2} \left( \int_{B_{kR}\setminus B_{2R}} |u(x)|^{2^{*}_{s}} dx\right)^{\frac{2}{2^{*}_{s}}} + C\left(\int_{B_{2R}\setminus B_{\frac{1}{k} R}} |u(x)|^{2^{*}_{s}} dx\right)^{\frac{2}{2^{*}_{s}}} \\
&\quad \quad + C\left(\frac{1}{k-1}\right)^{2s} \left(\int_{B_{\frac{1}{k} R}} |u(x)|^{2^{*}_{s}} dx\right)^{\frac{2}{2^{*}_{s}}}\, \Bigr]\\
&\quad\leq \lim_{k\rightarrow \infty} \frac{C}{k^{3}} + C\left(\frac{1}{k-1}\right)^{2s} \left(\int_{\R^{3}} |u(x)|^{2^{*}_{s}} dx \right)^{\frac{2}{2^{*}_{s}}}= 0
\end{align*}
Putting together \eqref{ter3}, \eqref{ter4}, \eqref{ter5} and by using the definition of $\eta_{R}$, we deduce that
\begin{equation*}
\lim_{R\rightarrow \infty} \limsup_{n\rightarrow \infty} \int_{\R^{3}\setminus B_{R}} \,dx \int_{\R^{3}} \frac{|u_{n}(x)- u_{n}(y)|^{2}}{|x-y|^{3+2s}}\,dy + \int_{\R^{3}\setminus B_{R}} V(\e x) u_{n}^{2} \,dx =0. 
\end{equation*}
This ends the proof of the Lemma. 
\end{proof}

\begin{lem}\label{lem2.6}
Let $\{u_{n}\}_{n\in\mathbb{N}}$ be a $(PS)_{d}$ sequence for $\mathcal{J}_{\e}$ such that $u_{n}\rightharpoonup u$. Then, for all $R>0$
\begin{align*}
\lim_{n\rightarrow \infty} &\int_{B_{R}} dx \int_{\R^{3}} \frac{|u_{n}(x)-u_{n}(y)|^{2}}{|x-y|^{3+2s}} \, dy + \int_{B_{R}} V(\e x) u_{n}^{2} dx\\
&= \int_{B_{R}} dx \int_{\R^{3}} \frac{|u(x)-u(y)|^{2}}{|x-y|^{3+2s}} \, dy + \int_{B_{R}} V(\e x) u^{2} dx.  
\end{align*}
\end{lem}

\begin{proof}
By Lemma \ref{lem2.4} we know that $\{u_{n}\}_{n\in\mathbb{N}}$ is bounded, so we may assume that $u_{n}\rightharpoonup u$ and
$\|u_{n}\|_{\e}\rightarrow t_{0}\geq 0$. Then, by the weak lower semicontinuity $\|u\|_{\e}\leq t_{0}$. \\
Let $\eta_{\rho}\in C^{\infty}(\R^{3})$ be such that $\eta_{\rho}=1$ in $B_{\rho}$ and $\eta_{\rho}=0$ in $B_{2\rho}^{c}$, with $0\leq \eta_{\rho}\leq 1$. 
Fix $R>0$ and choose $\rho>R$. Then we have 
\begin{align}\label{Phi}
&M(\|u_{n}\|_{\e}^{2})\left[\int_{B_{R}} dx \int_{\R^{3}} \frac{|(u_{n}(x) - u_{n}(y))- (u(x)- u(y))|^{2}}{|x-y|^{3+2s}} \, dy + \int_{B_{R}} V(\e x) (u_{n}-u)^{2} dx\right] \nonumber \\
&\leq M(\|u_{n}\|_{\e}^{2})\left[ \iint_{\R^{6}} \frac{|(u_{n}(x) - u_{n}(y))- (u(x)- u(y))|^{2}}{|x-y|^{3+2s}} \eta_{\rho}(x)\, dxdy + \int_{\R^{3}} V(\e x) (u_{n}- u)^{2}\eta_{\rho}(x)\,dx\right] \nonumber \\
&\leq M(\|u_{n}\|_{\e}^{2})\left[ \iint_{\R^{6}} \frac{|u_{n}(x) - u_{n}(y)|^{2}}{|x-y|^{3+2s}} \eta_{\rho}(x)\, dxdy + \int_{\R^{3}} V(\e x) u_{n}^{2}\eta_{\rho}(x)\,dx\right] \nonumber \\
&+ M(\|u_{n}\|_{\e}^{2})\left[ \iint_{\R^{6}} \frac{|u(x) - u(y)|^{2}}{|x-y|^{3+2s}} \eta_{\rho}(x)\, dxdy + \int_{\R^{3}} V(\e x) u^{2}\eta_{\rho}(x)\,dx\right] \nonumber \\
&-2M(\|u_{n}\|_{\e}^{2})\left[ \iint_{\R^{6}} \frac{(u_{n}(x) - u_{n}(y))(u(x)-u(y))}{|x-y|^{3+2s}} \eta_{\rho}(x)\, dxdy + \int_{\R^{3}} V(\e x) u_{n}\,u\,\eta_{\rho}(x)\,dx\right] \nonumber \\
&=I_{n,\rho}-II_{n,\rho}+III_{n,\rho}+IV_{n,\rho}\leq |I_{n,\rho}|+ |II_{n,\rho}|+ |III_{n,\rho}|+ |IV_{n,\rho}|
\end{align}
where 
\begin{align*}
&I_{n,\rho}\!\!:=\!\!M(\|u_{n}\|_{\e}^{2})\left[ \iint_{\R^{6}} \frac{|u_{n}(x) - u_{n}(y)|^{2}}{|x-y|^{3+2s}} \eta_{\rho}(x)\, dxdy + \int_{\R^{3}} V(\e x) u_{n}^{2}\eta_{\rho}(x)\,dx\right]- \int_{\R^{3}} g(\e x, u_{n})\, u_{n}\, \eta_{\rho}\, dx,\\
&II_{n,\rho}\!\!:=\!\!M(\|u_{n}\|_{\e}^{2})\left[ \iint_{\R^{6}} \frac{(u_{n}(x) - u_{n}(y))(u(x)-u(y))}{|x-y|^{3+2s}} \eta_{\rho}(x)\, dxdy + \int_{\R^{3}} V(\e x) u_{n}u \eta_{\rho}(x)\,dx\right]\\
&\quad \quad- \int_{\R^{3}} g(\e x, u_{n}) u \eta_{\rho}\, dx,\\
&III_{n,\rho}:=M(\|u_{n}\|_{\e}^{2})\left[ \iint_{\R^{6}} \frac{(u_{n}(x) - u_{n}(y))(u(x)-u(y))}{|x-y|^{3+2s}} \eta_{\rho}(x) dxdy + \int_{\R^{3}} V(\e x) u_{n}u \eta_{\rho}(x)\,dx\right]\\
&\quad \quad+M(\|u_{n}\|_{\e}^{2})\left[ \iint_{\R^{6}} \frac{|u(x) - u(y)|^{2}}{|x-y|^{3+2s}} \eta_{\rho}(x)\, dxdy + \int_{\R^{3}} V(\e x) u^{2}\eta_{\rho}(x)\,dx\right],\\
&IV_{n,\rho}:=\int_{\R^{3}} g(\e x, u_{n})\, u_{n}\, \eta_{\rho}\, dx- \int_{\R^{3}} g(\e x, u_{n})\, u\, \eta_{\rho}\, dx.
\end{align*}
Let us prove that 
\begin{equation}\label{I0}
\lim_{\rho \rightarrow \infty} \limsup_{n\rightarrow \infty} |I_{n,\rho}|=0. 
\end{equation}
Firstly, let us observe that $I_{n,\rho}$ can be written as
\begin{align*}
I_{n,\rho}=\langle \mathcal{J}_{\e}'(u_{n}), u_{n}\eta_{\rho} \rangle - M(\|u_{n}\|_{\e}^{2}) \iint_{\R^{6}} \frac{(u_{n}(x)- u_{n}(y))(\eta_{\rho}(x)- \eta_{\rho}(y))}{|x-y|^{3+2s}}u_{n}(y)\, dxdy.
\end{align*}
Since $\{u_{n}\eta_{\rho}\}_{n\in \mathbb{N}}$ is bounded in $\mathcal{H}_{\e}$, we have $\langle \mathcal{J}_{\e}'(u_{n}), u_{n}\eta_{\rho} \rangle=o_{n}(1)$, so 
\begin{equation}\label{vin1}
I_{n, \rho}= o_{n}(1)-M(\|u_{n}\|_{\e}^{2}) \iint_{\R^{6}} \frac{(u_{n}(x)- u_{n}(y))(\eta_{\rho}(x)- \eta_{\rho}(y))}{|x-y|^{3+2s}}u_{n}(y)\, dxdy.
\end{equation}
Arguing as in the proof of \eqref{ter5} (with $\eta_{R}= 1-\eta_{\rho}$), we can infer that
\begin{equation*}
\lim_{\rho \rightarrow \infty} \limsup_{n\rightarrow \infty} \left|M(\|u_{n}\|_{\e}^{2}) \iint_{\R^{6}} \frac{(u_{n}(x)- u_{n}(y))(\eta_{\rho}(x)- \eta_{\rho}(y))}{|x-y|^{3+2s}}u_{n}(x)\, dxdy\right| =0,
\end{equation*}
which together with \eqref{vin1} implies that \eqref{I0} holds. 
Now, we note that 
\begin{align*}
II_{n, \rho}=\langle \mathcal{J}_{\e}'(u_{n}), u\eta_{\rho} \rangle - M(\|u_{n}\|_{\e}^{2}) \iint_{\R^{6}} \frac{(u_{n}(x)- u_{n}(y))(\eta_{\rho}(x)- \eta_{\rho}(y))}{|x-y|^{3+2s}}u(x)\, dxdy. 
\end{align*}
Similar calculations to the proof of \eqref{ter5} show that
\begin{equation*}
\lim_{\rho \rightarrow \infty} \limsup_{n\rightarrow \infty} \left|M(\|u_{n}\|_{\e}^{2}) \iint_{\R^{6}} \frac{(u_{n}(x)- u_{n}(y))(\eta_{\rho}(x)- \eta_{\rho}(y))}{|x-y|^{3+2s}}u(x)\, dxdy\right| =0,
\end{equation*} 
and by using $\langle \mathcal{J}_{\e}'(u_{n}), u\eta_{\rho} \rangle=o_{n}(1)$, we obtain
\begin{equation}\label{II0}
\lim_{\rho \rightarrow \infty} \limsup_{n\rightarrow \infty} |II_{n, \rho}|=0. 
\end{equation}
On the other hand, from the weak convergence, we have
\begin{equation}\label{III0}
\lim_{n\rightarrow \infty} |III_{n, \rho}|=0, 
\end{equation}
for any $\rho>R$.
By using Theorem \ref{Sembedding}, we know that $u_{n}\rightarrow u$ in $L^{p}_{loc}(\R^{3})$ for $2\leq p<\frac{6}{3-2s}$.
Hence, in view of $(g_{1})$ and $(g_{2})$, we deduce that for any $\rho>R$
\begin{equation}\label{IV0}
\lim_{n\rightarrow \infty} |IV_{n, \rho}|=0. 
\end{equation}
Putting together \eqref{Phi}, \eqref{I0}, \eqref{II0}, \eqref{III0} and \eqref{IV0}, and recalling that $\|u_{n}\|_{\e}\rightarrow t_{0}$ we get the thesis. 
\end{proof}

\noindent
Taking into account the previous lemmas, we can demonstrate the following result.
\begin{prop}\label{prop2.2}
The functional $\mathcal{J}_{\e}$ verifies the $(PS)_{d}$ condition in $\mathcal{H}_{\e}$.
\end{prop}
 
\begin{proof}
Let $\{u_{n}\}_{n\in\mathbb{N}}$ be a $(PS)$ sequence for $\mathcal{J}_{\e}$ at the level $d$. 
By Lemma \ref{lem2.4} we know that $\{u_{n}\}_{n\in\mathbb{N}}$ is bounded in $\mathcal{H}_{\e}$, thus, up to a subsequence, we deduce
\begin{equation}\label{ter6}
u_{n}\rightharpoonup u \mbox{ in } \mathcal{H}_{\e}.
\end{equation}
By using Lemma \ref{lem2.6} we know that
\begin{align}\begin{split}\label{ter7}
\lim_{n\rightarrow \infty} &\int_{B_{R}} dx \int_{\R^{3}} \frac{|u_{n}(x)-u_{n}(y)|^{2}}{|x-y|^{3+2s}} \, dy + \int_{B_{R}} V(\e x) u_{n}^{2} dx\\
&= \int_{B_{R}} dx \int_{\R^{3}} \frac{|u(x)-u(y)|^{2}}{|x-y|^{3+2s}} \, dy + \int_{B_{R}} V(\e x) u^{2} dx.  
\end{split}\end{align}
Moreover, by Lemma \ref{lem2.5} for each $\zeta>0$, there exists $R=R(\zeta)>\frac{C}{\zeta}$ such that 
\begin{equation}\label{ter8}
\limsup_{n\rightarrow \infty} \left[ \int_{\R^{3}\setminus B_{R}} \,dx \int_{\R^{3}} \frac{|u_{n}(x)- u_{n}(y)|^{2}}{|x-y|^{3+2s}}\,dy + \int_{\R^{3}\setminus B_{R}} V(\e x) u_{n}^{2} \,dx\right] <\zeta. 
\end{equation}
Putting together \eqref{ter6}, \eqref{ter7} and \eqref{ter8} we can infer
\begin{align*}
\|u\|_{\e}^{2} &\leq \liminf_{n\rightarrow \infty} \|u_{n}\|_{\e}^{2} \leq  \limsup_{n\rightarrow \infty} \|u_{n}\|_{\e}^{2}  \\
&= \limsup_{n\rightarrow \infty} \Bigl[ \, \int_{B_{R}} dx \int_{\R^{3}} \frac{|u_{n}(x)- u_{n}(y)|^{2}}{|x-y|^{3+2s}}\, dy + \int_{B_{R}} V(\e x) \,u_{n}^{2} dx \\
&\quad\quad+  \int_{\R^{3}\setminus B_{R}} dx \int_{\R^{3}} \frac{|u_{n}(x)- u_{n}(y)|^{2}}{|x-y|^{3+2s}}\, dy + \int_{\R^{3}\setminus B_{R}} V(\e x) \,u_{n}^{2} dx \, \Bigr]\\
&\leq \int_{B_{R}} dx \int_{\R^{3}} \frac{|u(x)-u(y)|^{2}}{|x-y|^{3+2s}} \, dy + \int_{B_{R}} V(\e x) u^{2} dx +\zeta.
\end{align*}
Taking the limit as $\zeta\rightarrow 0$, we have $R\rightarrow \infty$, therefore
\begin{equation*}
\|u\|_{\e}^{2}\leq \liminf_{n\rightarrow \infty} \|u_{n}\|_{\e}^{2} \leq \limsup_{n\rightarrow \infty} \|u_{n}\|_{\e}^{2}\leq \|u\|_{\e}^{2}, 
\end{equation*}
which implies $\|u_{n}\|_{\e}\rightarrow \|u\|_{\e}$. Since $\mathcal{H}_{\e}$ is a Hilbert space, we can deduce $u_{n}\rightarrow u$ in $\mathcal{H}_{\e}$. 
\end{proof}

\begin{cor}\label{cor2.1}
The functional $\psi_{\e}$ verifies the $(PS)_{d}$ condition on $\mathbb{S}_{\e}^{+}$. 
\end{cor}

\begin{proof}
Let $\{u_{n}\}_{n\in \mathbb{N}}$ be a $(PS)$ sequence for $\psi_{\e}$ at the level $d$. Then
\begin{equation*}
\psi_{\e}(u_{n})\rightarrow d \quad \mbox{ and } \quad \psi_{\e}'(u_{n})\rightarrow 0 \mbox{ in } (T_{u_{n}} \mathbb{S}_{\e}^{+})'.
\end{equation*}
From Proposition \ref{prop2.1}-$(c)$ follows that $\{m_{\e}(u_{n})\}_{n\in \mathbb{N}}$ is a $(PS)_{d}$ sequence for $\mathcal{J}_{\e}$ in $\mathcal{H}_{\e}$. Then, by using Proposition \ref{prop2.2} we see that $\mathcal{J}_{\e}$ verifies the $(PS)_{d}$ condition in $\mathcal{H}_{\e}$, so there exists $u\in \mathbb{S}_{\e}^{+}$ such that, up to a subsequence, 
\begin{equation*}
m_{\e}(u_{n})\rightarrow m_{\e}(u) \mbox{ in } \mathcal{H}_{\e}. 
\end{equation*}
By applying Lemma \ref{lem2.3}-$(iii)$ we can infer that $u_{n}\rightarrow u$ in $\mathbb{S}_{\e}^{+}$.
\end{proof}

\noindent
At this point, we are able to prove the main result of this Section. 
\begin{proof}[Proof of Theorem \ref{thm2.1}]
In view of Lemma \ref{lem2.2} and Proposition \ref{prop2.2} we can apply the Mountain Pass Theorem \cite{AR}, so we obtain the existence of a nontrivial critical point $u_{\e}$ of $\mathcal{J}_{\e}$. Now, we show that $u_{\e}\geq 0$ in $\R^{3}$.
Since $\langle \mathcal{J}_{\e}'(u_{\e}), u^{-}_{\e}\rangle =0$, we can see that
\begin{align*}
M(\|u_{\e}\|_{\e}^{2}) \left[\iint_{\R^{6}} \frac{(u_{\e}(x)-u_{\e}(y))(u^{-}_{\e}(x)-u^{-}_{\e}(y))}{|x-y|^{3-2s}} \, dx dy+\int_{\R^{3}} V(\e x) u_{\e} u_{\e}^{-}\, dx\right]=\int_{\R^{3}} g(\e x, u_{\e}) u_{\e}^{-} \, dx
\end{align*} 
Recalling that $(x-y)(x^{-}- y^{-})\leq - |x^{-}-y^{-}|^{2}$ and $g(x, t)=0$ for $t\leq 0$, we deduce that 
\begin{equation*}
0\leq M(\|u_{\e}\|_{\e}^{2}) \|u^{-}_{\e}\|_{\e}^{2} \leq 0.
\end{equation*}
By using $(M_1)$, we have $\|u^{-}_{\e}\|_{\e}^{2}=0$ that is $u_{\e}\geq 0$ in $\R^{3}$. 
\end{proof}

\section{The autonomous problem}
\noindent
In this section we deal with the limit problem associated to \eqref{Pe}.
More precisely, we consider the following problem 
\begin{equation}\label{P0}
M\left(\iint_{\R^{6}}\frac{|u(x)- u(y)|^{2}}{|x-y|^{3+2s}} dxdy + \int_{\R^{3}} V_{0}u^{2} dx\right)[(-\Delta)^{s}u+ V_{0}u]= f(u) \, \mbox{ in } \R^{3}. 
\end{equation}
The Euler-Lagrange functional associated to \eqref{P0} is 
\begin{equation*}
\mathcal{J}_{0}(u)= \frac{1}{2}\widehat{M}\left(\iint_{\R^{6}}\frac{|u(x)- u(y)|^{2}}{|x-y|^{3+2s}} dxdy + \int_{\R^{3}} V_{0}u^{2} dx\right)- \int_{\R^{3}} F(u)\, dx
\end{equation*}
which is well defined on the Hilbert space $\mathcal{H}_{0}:=H^{s}(\R^{3})$ endowed with the inner product
\begin{equation*}
(u, \varphi)_{0} = \iint_{\R^{6}} \frac{(u(x)- u(y))(\varphi(x) - \varphi(y))}{|x-y|^{3+2s}} dxdy + \int_{\R^{3}} V_{0} u(x)\, \varphi(x) dx. 
\end{equation*}
The norm induced by the inner product is 
\begin{equation*}
\|u\|_{0}^{2} = \iint_{\R^{6}}\frac{|u(x)- u(y)|^{2}}{|x-y|^{3+2s}} dxdy + \int_{\R^{3}} V_{0} u^{2} dx. 
\end{equation*}
The Nehari manifold associated to $\mathcal{J}_{0}$ is given by 
\begin{equation*}
\N_{0}= \{u\in \mathcal{H}_{0}\setminus \{0\} : \langle \mathcal{J}_{0}'(u), u \rangle=0\}. 
\end{equation*}
We denote by $\mathcal{H}_{0}^{+}$ the open subset of $\mathcal{H}_{0}$ defined as
\begin{equation*}
\mathcal{H}_{0}^{+}=\{u\in \mathcal{H}_{0}: |\supp(u^{+})|>0\},
\end{equation*}
and $\mathbb{S}_{0}^{+}= \mathbb{S}_{0}\cap \mathcal{H}_{0}^{+}$, where $\mathbb{S}_{0}$ is the unit sphere of $\mathcal{H}_{0}$. We note that $\mathbb{S}_{0}^{+}$ is a incomplete $C^{1,1}$-manifold of codimension $1$ modeled on $\mathcal{H}_{0}$ and contained in $\mathcal{H}_{0}^{+}$. Thus $\mathcal{H}_{0}= T_{u}\mathbb{S}_{0}^{+}\oplus \R u$ for each $u\in \mathbb{S}_{0}^{+}$, where $T_{u}\mathbb{S}_{0}^{+}=\{u\in \mathcal{H}_{0} : (u, v)_{0}=0\}$. 

\noindent
As in Section $3$, we can see that the following results hold.
\begin{lem}\label{lem3.1}
Assume that $(M_1)-(M_3)$ and $(f_1)-(f_4)$ hold true. Then, 
\begin{compactenum}[$(i)$]
\item For each $u\in \mathcal{H}_{0}^{+}$, let $h:\R^{+}\rightarrow \R$ be defined by $h_{u}(t)= \mathcal{J}_{0}(tu)$. Then, there is a unique $t_{u}>0$ such that 
\begin{align*}
&h'_{u}(t)>0 \mbox{ in } (0, t_{u})\\
&h'_{u}(t)<0 \mbox{ in } (t_{u}, \infty); 
\end{align*}
\item there exists $\tau>0$ independent of $u$ such that $t_{u}\geq \tau$ for any $u\in \mathbb{S}_{0}^{+}$. Moreover, for each compact set $\mathbb{K}\subset \mathbb{S}_{0}^{+}$ there is a positive constant $C_{\mathbb{K}}$ such that $t_{u}\leq C_{\mathbb{K}}$ for any $u\in \mathbb{K}$; 
\item The map $\hat{m}_{0}: \mathcal{H}_{0}^{+}\rightarrow \N_{0}$ given by $\hat{m}_{0}(u)= t_{u}u$ is continuous and $m_{0}:= \hat{m}_{0}|_{\mathbb{S}_{0}^{+}}$ is a homeomorphism between $\mathbb{S}_{0}^{+}$ and $\N_{0}$. Moreover $m_{0}^{-1}(u)=\frac{u}{\|u\|_{0}}$; 
\item If there is a sequence $\{u_{n}\}_{n\in \mathbb{N}}\subset \mathbb{S}_{0}^{+}$ such that ${\rm dist}(u_{n}, \partial \mathbb{S}_{0}^{+})\rightarrow 0$ then $\|m_{0}(u_{n})\|_{0}\rightarrow \infty$ and $\mathcal{J}_{0}(m_{0}(u_{n}))\rightarrow \infty$.
\end{compactenum}

\end{lem}

\noindent
Let us define the maps 
\begin{equation*}
\hat{\psi}_{0}: \mathcal{H}_{0}^{+} \rightarrow \R \quad \mbox{ and } \quad \psi_{0}: \mathbb{S}_{0}^{+}\rightarrow \R, 
\end{equation*}
by $\hat{\psi}_{0}(u):= \mathcal{J}_{0}(\hat{m}_{0}(u))$ and $\psi_{0}:=\hat{\psi}_{0}|_{\mathbb{S}_{0}^{+}}$. 
\begin{prop}\label{prop3.1}
Assume that assumptions $(M_{1})$-$(M_{3})$ and $(f_{1})$-$(f_{4})$ hold true. Then, 
\begin{compactenum}[$(a)$]
\item $\hat{\psi}_{0} \in C^{1}(\mathcal{H}_{0}^{+}, \R)$ and 
\begin{equation*}
\langle \hat{\psi}_{0}'(u), v\rangle = \frac{\|\hat{m}_{0}(u)\|_{0}}{\|u\|_{0}} \langle \mathcal{J}_{0}'(\hat{m}_{0}(u)), v\rangle 
\end{equation*}
for every $u\in \mathcal{H}_{0}^{+}$ and $v\in \mathcal{H}_{0}$; 
\item $\psi_{0} \in C^{1}(\mathbb{S}_{0}^{+}, \R)$ and 
\begin{equation*}
\langle \psi_{0}'(u), v \rangle = \|m_{0}(u)\|_{0} \langle \mathcal{J}_{0}'(m_{0}(u)), v\rangle, 
\end{equation*}
for every 
\begin{equation*}
v\in T_{u}\mathbb{S}_{0}^{+}:= \left \{ v\in \mathcal{H}_{0} : (u, v)_{0}=0 \right\}; 
\end{equation*}
\item If $\{u_{n}\}_{n\in \mathbb{N}}$ is a $(PS)_{d}$ sequence for $\psi_{0}$, then $\{m_{0}(u_{n})\}_{n\in \mathbb{N}}$ is a $(PS)_{d}$ sequence for $\mathcal{J}_{0}$. If $\{u_{n}\}_{n\in \mathbb{N}}\subset \N_{0}$ is a bounded $(PS)_{d}$ sequence for $\mathcal{J}_{0}$, then $\{m_{0}^{-1}(u_{n})\}_{n\in \mathbb{N}}$ is a $(PS)_{d}$ sequence for the functional $\psi_{0}$; 
\item $u$ is a critical point of $\psi_{0}$ if, and only if, $m_{0}(u)$ is a nontrivial critical point for $\mathcal{J}_{0}$. Moreover, the corresponding critical values coincide and 
\begin{equation*}
\inf_{u\in \mathbb{S}_{0}^{+}} \psi_{0}(u)= \inf_{u\in \N_{0}} \mathcal{J}_{0}(u).  
\end{equation*}
\end{compactenum}
\end{prop}

\begin{remark}\label{rem4}
As in Section $3$, we have the following equalities
\begin{align*}
c_{0}&:=\inf_{u\in \N_{0}} \mathcal{J}_{0}(u)=\inf_{u\in \mathcal{H}_{0}^{+}} \max_{t>0} \mathcal{J}_{0}(tu)=\inf_{u\in \mathbb{S}_{0}^{+}} \max_{t>0} \mathcal{J}_{0}(tu).
\end{align*}
\end{remark}

\noindent
The following Lemma is very important because permits to deduce that the weak limit of a $(PS)_{d}$ sequence is nontrivial. 
\begin{lem}\label{lem3.2}
Let $\{u_{n}\}_{n\in \mathbb{N}}\subset \mathcal{H}_{0}$ be a $(PS)_{d}$ sequence for $\mathcal{J}_{0}$ with $u_{n}\rightharpoonup 0$. Then, only one of the alternative below holds:
\begin{compactenum}[$(a)$]
\item $u_{n}\rightarrow 0$ in $\mathcal{H}_{0}$; 
\item there are a sequence $\{y_{n}\}_{n\in \mathbb{N}}\subset \R^{3}$ and constants $R, \beta>0$ such that
\begin{equation*}
\liminf_{n\rightarrow \infty} \int_{B_{R}(y_{n})} u_{n}^{2} dx \geq \beta >0. 
\end{equation*}
\end{compactenum}
\end{lem}

\begin{proof}
Assume that $(b)$ does not hold. Then, for  all $R>0$ we have
\begin{equation*}
\lim_{n\rightarrow \infty} \sup_{y\in \R^{3}} \int_{B_{R}(y)} u_{n}^{2} dx=0. 
\end{equation*}
Since $\{u_{n}\}_{n\in \mathbb{N}}$ is bounded in $\mathcal{H}_{0}$, from Lemma \ref{Lions} follows that
\begin{equation*}
u_{n}\rightarrow 0 \, \mbox{ in } L^{p}(\R^{3}) \, \mbox{ for any } 2<p<2^{*}_{s}. 
\end{equation*}
Moreover, by $(f_{1})$ and $(f_{2})$, we can see that 
$$
\lim_{n\rightarrow \infty} \int_{\R^{3}} f(u_{n}) u_{n} \, dx=0.
$$
Then, by using $\langle \mathcal{J}'_{0}(u_{n}), u_{n}\rangle=o_{n}(1)$ and $(M_1)$, we can see that
\begin{equation*}
0\leq m_{0}\|u_{n}\|^{2}_{0}\leq M(\|u_{n}\|^{2}_{0})\|u_{n}\|^{2}_{0}= \int_{\R^{3}} f(u_{n})u_{n} \, dx +o_{n}(1)= o_{n}(1). 
\end{equation*}
Therefore it holds $(a)$. 
\end{proof}

\begin{remark}\label{rem5}
Let us observe that, if $\{u_{n}\}_{n\in \mathbb{N}}$ is a $(PS)$ sequence at the level $c_{0}$ for the functional $\mathcal{J}_{0}$ such that $u_{n}\rightharpoonup u$, then $u\neq 0$. Otherwise, if $u_{n}\rightharpoonup 0$ and, once it does not occur $u_{n}\rightarrow 0$, in view of Lemma \ref{lem3.2} we can find $\{y_{n}\}_{n\in \mathbb{N}}\subset \R^{3}$ and $R, \beta>0$ such that
\begin{equation*}
\liminf_{n\rightarrow \infty} \int_{B_{R}(y_{n})} u_{n}^{2} dx \geq \beta >0. 
\end{equation*}
Set $v_{n}(x)=u_{n}(x+y_{n})$, and making a change of variable, we can see that $\{v_{n}\}_{n\in \mathbb{N}}$ is a $(PS)$ sequence for $\mathcal{J}_{0}$ at the level $c_{0}$, $\{v_{n}\}_{n\in \mathbb{N}}$ is bounded in $\mathcal{H}_{0}$ and there exists $v\in \mathcal{H}_{0}$ such that $v_{n}\rightharpoonup v$ and $v\neq 0$. 
\end{remark}

\begin{thm}\label{thm3.1}
The problem \eqref{P0} admits a positive solution. 
\end{thm}

\begin{proof}
Arguing as in the proof of Lemma \ref{lem2.2}, it is easy to check that $\mathcal{J}_{0}$ has a mountain pass geometry. By using Theorem $1.15$ in \cite{Willem}, we know that there exists a Palais-Smale sequence $\{u_{n}\}_{n\in \mathbb{N}}$ for $\mathcal{J}_{0}$ at the level $c_{0}$, that is
\begin{equation*}
\mathcal{J}_{0}(u_{n})\rightarrow c_{0} \quad \mbox{ and } \quad \mathcal{J}_{0}'(u_{n})\rightarrow 0 \mbox{ in } \mathcal{H}_{0}^{-1}. 
\end{equation*}
Let us observe that $\{u_{n}\}_{n\in\mathbb{N}}$ is a bounded sequence in $\mathcal{H}_{0}$. Indeed, by using \eqref{ccc}, assumptions $(f_{3})$ and $(M_{1})$, and taking into account that $\vartheta>4$, we have
\begin{align*}
C+\|u_{n}\|_{0}&\geq \mathcal{J}_{0}(u_{n})-\frac{1}{\vartheta} \langle \mathcal{J}_{0}'(u_{n}), u_{n}\rangle \\
&=\frac{1}{2}\widehat{M}(\|u_{n}\|_{0}^{2})- \frac{1}{\vartheta} M(\|u_{n}\|_{0}^{2})\|u_{n}\|_{0}^{2} + \frac{1}{\vartheta}\int_{\R^{3}} [f(u)u-F(u)] \, dx\\
&\geq \left(\frac{1}{4}- \frac{1}{\vartheta}\right) M(\|u_{n}\|_{0}^{2})\|u_{n}\|_{0}^{2} + \frac{m_{0}}{4} \|u_{n}\|_{0}^{2}\\
&\geq \left( \frac{\vartheta-2}{2\vartheta}\right)m_{0}\|u_{n}\|_{0}^{2},
\end{align*}
which yields the boundedness of $\{u_{n}\}_{n\in\mathbb{N}}$ being $\vartheta>4$.
Therefore, in view of Theorem \ref{Sembedding}, we may assume that
\begin{align}
&u_{n}\rightharpoonup u \mbox{ in } \mathcal{H}_{0}, \label{ter10}\\
&u_{n}\rightarrow u \mbox{ in } L^{p}_{loc}(\R^{3}) \mbox{ for all } 2\leq p< \frac{6}{3-2s}, \label{ter11}\\
&\|u_{n}\|_{0}\rightarrow t_{0} \label{ter12}. 
\end{align}
We recall that $\langle \mathcal{J}_{0}'(u_{n}), \varphi \rangle =o_{n}(1)$ for any $\varphi \in \mathcal{H}_{0}$, that is
\begin{equation}\label{vin2}
M(\|u_{n}\|^{2}_{0})(u_{n}, \varphi)_{0}= \int_{\R^{3}} f(u_{n})\varphi \,dx+ o_{n}(1). 
\end{equation}
Take $\varphi\in C^{\infty}_{c}(\R^{3})$. From \eqref{ter10} we have 
\begin{equation}\label{vin3}
(u_{n}, \varphi)_{0}\rightarrow (u, \varphi)_{0}.
\end{equation}
On the other hand, from $(f_{1})$, $(f_{2})$ and \eqref{ter11} we get
\begin{equation}\label{vin4}
\int_{\R^{3}} f(u_{n}) \,  \varphi\, dx\rightarrow \int_{\R^{3}} f(u)\,  \varphi\, dx. 
\end{equation}
In order to prove that $u$ is a weak solution to \eqref{P0}, it remains to prove that 
\begin{equation}\label{vin5}
M(\|u_{n}\|_{0}^{2})\rightarrow M(t_{0}^{2}).
\end{equation} 
From \eqref{ter12}, we can note that
\begin{equation*}
\|u\|_{0}^{2}\leq \liminf_{n\rightarrow \infty}\|u_{n}\|_{0}^{2}= t_{0}^{2}, 
\end{equation*}
so, by using $(M_{2})$ we deduce that $M(\|u\|_{0}^{2})\leq M(t_{0}^{2})$. At this point our aim is to prove that $M(\|u\|_{0}^{2})= M(t_{0}^{2})$. Assume by contradiction that $M(\|u\|_{0}^{2})< M(t_{0}^{2})$, then
\begin{equation*}
M(\|u\|_{0}^{2})\|u\|_{0}^{2}< M(t_{0}^{2})\|u\|_{0}^{2} = \int_{\R^{3}} f(u)\, u\, dx,
\end{equation*}
that is $\langle \mathcal{J}_{0}'(u), u\rangle<0$. Then, there exists $t\in (0,1)$ such that $t u\in \N_{0}$. Now, by using the assumption $(M_{3})$ and \eqref{rem1} we have 
\begin{align*}
c_{0}&\leq \mathcal{J}_{0}(tu) = \mathcal{J}_{0}(tu)-\frac{1}{4}\langle \mathcal{J}_{0}'(tu), tu \rangle \\
&= \frac{1}{2}\widehat{M}(\|tu\|_{0}^{2}) - \frac{1}{4} M(\|tu\|_{0}^{2})\|tu\|_{0}^{2} + \int_{\R^{3}} \left[\frac{1}{4} f(tu)\,tu-F(tu)\right]\, dx\\
&<\frac{1}{2}\widehat{M}(\|u\|_{0}^{2}) - \frac{1}{4} M(\|u\|_{0}^{2})\|u\|_{0}^{2} + \int_{\R^{3}} \left[\frac{1}{4} f(u)\,u-F(u)\right]\, dx\\
&\leq \liminf_{n\rightarrow \infty} \left\{\frac{1}{2}\widehat{M}(\|u_{n}\|_{0}^{2}) - \frac{1}{4} M(\|u_{n}\|_{0}^{2})\|u_{n}\|_{0}^{2} + \int_{\R^{3}} \left[\frac{1}{4} f(u_{n})\,u_{n}-F(u_{n})\right]\, dx\right\}\\
&= \liminf_{n\rightarrow \infty} \left\{\mathcal{J}_{0}(u_{n})-\frac{1}{4}\langle \mathcal{J}_{0}'(u_{n}), u_{n} \rangle\right\} =c_{0}
\end{align*}
and this is a contradiction. Therefore, putting together \eqref{vin2}, \eqref{vin3}, \eqref{vin4} and \eqref{vin5} we obtain 
\begin{equation*}
M(\|u\|^{2}_{0})(u, \varphi)_{0}= \int_{\R^{3}} f(u)\varphi \,dx \, \mbox{ for any } \varphi \in C^{\infty}_{c}(\R^{3}). 
\end{equation*}
Since $C^{\infty}_{c}(\R^{3})$ is dense in $\mathcal{H}_{0}$, we deduce that $\mathcal{J}'_{0}(u)=0$.\\
Now, we show that $u>0$ in $\R^{3}$. Firstly we prove that $u\geq 0$. Indeed, observing that $\langle \mathcal{J}_{0}'(u), u^{-}\rangle =0$ and by using $(x-y)(x^{-}- y^{-})\leq - |x^{-}-y^{-}|^{2}$ and $f(t)=0$ for $t\leq 0$, we have 
\begin{equation*}
0\leq M(\|u_{0}\|_{0}^{2}) \|u^{-}\|_{0}^{2} \leq 0,
\end{equation*}
which together with $(M_1)$ implies $u\geq 0$.  Next, we show that $u\in C^{0, \alpha}(\R^{3})\cap L^{\infty}(\R^{3})$ for some $\alpha\in (0, 1)$.
Let 
$$
v=\frac{u}{M(\|u\|^{2}_{0})^{\frac{1}{q-2}}}.
$$
Then $v$ is a nonnegative solution to 
$$
(-\Delta)^{s} v+V_{0} v=h(v) \mbox{ in } \R^{3}
$$
where $h$ is the following continuous function
$$
h(v)=\frac{f(v M(\|u\|^{2}_{0})^{\frac{1}{q-2}})}{M(\|u\|^{2}_{0})^{\frac{q-1}{q-2}}}.
$$
By using $(f_3)$, we can see that 
$$
\lim_{t\rightarrow \infty} \frac{h(t)}{t^{q-1}}=\lim_{t\rightarrow \infty} \frac{f(t M(\|u\|^{2}_{0})^{\frac{1}{q-2}})}{\left(t M(\|u\|^{2}_{0})^{\frac{1}{q-2}}\right)^{q-1}}=0,
$$
so we deduce that $h(t)\leq C(1+|t|^{q-1})$ for any $t\in \R$. Then, by applying Theorem $2.3$ in \cite{DMPV}, we deduce that $v\in L^{\infty}(\R^{3})$. In view of Proposition $2.9$ in \cite{S}, we obtain that $v\in C^{0, \beta}(\R^{3})$ for some $\beta\in (0, 1)$. From the Harnack inequality \cite{CabS} we get $v>0$ in $\R^{3}$. Therefore $u\in C^{0, \alpha}(\R^{3})\cap L^{\infty}(\R^{3})$ is a positive solution \eqref{P0} and this ends the proof of theorem.
\end{proof}

\noindent
The next result is a compactness result on the autonomous problem which we will use later.
\begin{lem}\label{lem3.3}
Let $\{u_{n}\}_{n\in \mathbb{N}}\subset \N_{0}$ be a sequence such that $\mathcal{J}_{0}(u_{n})\rightarrow c_{0}$. Then  
$\{u_{n}\}_{n\in \mathbb{N}}$ has a convergent subsequence in $H^{s}(\R^{3})$.
\end{lem}
\begin{proof}
Since $\{u_{n}\}_{n\in \mathbb{N}}\subset \mathcal{N}_{0}$ and $\mathcal{J}_{0}(u_{n})\rightarrow c_{0}$, we can apply Lemma \ref{lem3.1}-$(iii)$ and Proposition \ref{prop3.1}-$(d)$ and Remark \ref{rem4} to infer that
$$
v_{n}=m^{-1}(u_{n})=\frac{u_{n}}{\|u_{n}\|_{0}}\in \mathbb{S}_{0}^{+}
$$
and
$$
\psi_{0}(v_{n})=\mathcal{J}_{0}(u_{n})\rightarrow c_{0}=\inf_{v\in \mathbb{S}_{0}^{+}}\psi_{0}(v).
$$
Let us introduce the following map $\mathcal{F}: \overline{\mathbb{S}}_{0}^{+}\rightarrow \R\cup \{\infty\}$ defined by setting
$$
\mathcal{F}(u):=
\begin{cases}
\psi_{0}(u)& \text{ if $u\in \mathbb{S}_{0}^{+}$} \\
\infty   & \text{ if $u\in \partial \mathbb{S}_{0}^{+}$}.
\end{cases}
$$ 
We note that 
\begin{itemize}
\item $(\overline{\mathbb{S}}_{0}^{+}, d_{0})$, where $d(u, v)=\|u-v\|_{0}$, is a complete metric space;
\item $\mathcal{F}\in C(\overline{\mathbb{S}}_{0}^{+}, \R\cup \{\infty\})$, by Lemma \ref{lem3.1}-$(iii)$;
\item $\mathcal{F}$ is bounded below, by Proposition \ref{prop3.1}-$(d)$.
\end{itemize}
Hence, by applying the Ekeland's variational principle \cite{Ekeland} to $\mathcal{F}$, we can find $\{\hat{v}_{n}\}_{n\in \mathbb{N}}\subset \mathbb{S}_{0}^{+}$ such that $\{\hat{v}_{n}\}_{n\in \mathbb{N}}$ is a $(PS)_{c_{0}}$ sequence for $\psi_{0}$ on $\mathbb{S}_{0}^{+}$ and $\|\hat{v}_{n}-v_{n}\|_{0}=o_{n}(1)$.
Then, by using Proposition \ref{prop3.1}, Theorem \ref{thm3.1} and arguing as in the proof of Corollary \ref{cor2.1} we obtain the thesis.
\end{proof}

\section{Barycenter map and multiplicity of solutions to \eqref{Pea}}

\noindent
In this section, our main purpose is to apply the Ljusternik-Schnirelmann category theory to prove a multiplicity result for the problem \eqref{Pea}. We begin by proving the following technical results.
\begin{lem}\label{lem3.1N}
Let $\e_{n}\rightarrow 0^{+}$ and $\{u_{n}\}_{n\in \mathbb{N}}\subset \N_{\e_{n}}$ be such that $\mathcal{J}_{\e_{n}}(u_{n})\rightarrow c_{0}$. Then there exists $\{\tilde{y}_{n}\}_{n\in \mathbb{N}}\subset \R^{3}$ such that the translated sequence 
\begin{equation*}
\tilde{u}_{n}(x):=u_{n}(x+ \tilde{y}_{n})
\end{equation*}
has a subsequence which converges in $H^{s}(\R^{3})$. Moreover, up to a subsequence, $\{y_{n}\}_{n\in \mathbb{N}}:=\{\e_{n}\tilde{y}_{n}\}_{n\in \mathbb{N}}$ is such that $y_{n}\rightarrow y_{0}\in \Lambda$. 
\end{lem}

\begin{proof}
Since $\langle \mathcal{J}'_{\e_{n}}(u_{n}), u_{n} \rangle=0$ and $\mathcal{J}_{\e_{n}}(u_{n})\rightarrow c_{0}$, it is easy to see that $\{u_{n}\}_{n\in \mathbb{N}}$ is bounded. 
Let us observe that $\|u_{n}\|_{\e_{n}}\nrightarrow 0$ since $c_{0}>0$. Therefore, arguing as in Remark \ref{rem5}, we can find a sequence $\{\tilde{y}_{n}\}_{n\in \mathbb{N}}\subset \R^{3}$ and constants $R, \alpha>0$ such that
\begin{equation*}
\liminf_{n\rightarrow \infty}\int_{B_{R}(\tilde{y}_{n})} |u_{n}|^{2} dx\geq \alpha.
\end{equation*}
Set $\tilde{u}_{n}(x):=u_{n}(x+ \tilde{y}_{n})$. Then it is clear that $\{\tilde{u}_{n}\}_{n\in \mathbb{N}}$ is bounded in $H^{s}(\R^{3})$, and we may assume that 
\begin{equation*}
\tilde{u}_{n}\rightharpoonup \tilde{u} \mbox{ weakly in } H^{s}(\R^{3}),  
\end{equation*}
for some $\tilde{u}\neq 0$.\\
Let $\{t_{n}\}_{n\in \mathbb{N}}\subset (0, +\infty)$ be such that $\tilde{v}_{n}:=t_{n}\tilde{u}_{n} \in \N_{0}$ (see Lemma \ref{lem3.1}-$(i)$), and set $y_{n}:=\e_{n}\tilde{y}_{n}$.  \\
Then, by using $(M_{2})$ and $g(x, t)\leq f(t)$, we can see that
\begin{align*}
c_{0}\leq \mathcal{J}_{0}(\tilde{v}_{n})&= \frac{1}{2} \widehat{M}(t_{n}^{2}\|u_{n}\|^{2}_{0}) - \int_{\R^{3}} F(t_{n} u_{n})\, dx \nonumber\\
&\leq \frac{1}{2} \widehat{M}(t_{n}^{2}\|u_{n}\|^{2}_{\e_{n}}) - \int_{\R^{N}} G(\e x, t_{n} u_{n})\, dx \nonumber\\
&=\mathcal{J}_{\e_{n}}(t_{n}u_{n}) \leq \mathcal{J}_{\e_{n}}(u_{n})= c_{0}+ o_{n}(1),
\end{align*}
which gives 
\begin{align}\label{3.21}
\mathcal{J}_{0}(\tilde{v}_{n})\rightarrow c_{0} \,\mbox{ and } \,\{\tilde{v}_{n}\}_{n\in \mathbb{N}}\subset \N_{0}. 
\end{align}
In particular, \eqref{3.21} implies that $\{\tilde{v}_{n}\}_{n\in \mathbb{N}}$ is bounded in $H^{s}(\R^{3})$, so we may assume that $\tilde{v}_{n}\rightharpoonup \tilde{v}$. Obviously, $\{t_{n}\}_{n\in \mathbb{N}}$ is bounded and it results $t_{n}\rightarrow t_{0}\geq 0$. If $t_{0}=0$, from the boundedness of $\{\tilde{u}_{n}\}_{n\in \mathbb{N}}$, we get $\|\tilde{v}_{n}\|_{0}= t_{n}\|\tilde{u}_{n}\|_{0} \rightarrow 0$, that is $\mathcal{J}_{0}(\tilde{v}_{n})\rightarrow 0$ in contrast with the fact $c_{0}>0$. Then, $t_{0}>0$. From the uniqueness of the weak limit we have $\tilde{v}=t_{0} \tilde{u}$ and $\tilde{u}\neq 0$. By using Lemma \ref{lem3.3} we deduce that 
\begin{align}\label{3.22}
\tilde{v}_{n}\rightarrow \tilde{v} \mbox{ in } H^{s}(\R^{3}),
\end{align}
which implies that $\displaystyle{\tilde{u}_{n}=\frac{\tilde{v}_{n}}{t_{n}}\rightarrow \frac{\tilde{v}}{t_{0}}=\tilde{u}}$ in $H^{s}(\R^{3})$ and
\begin{equation*}
\mathcal{J}_{0}(\tilde{v})=c_{0} \, \mbox{ and } \, \langle\mathcal{J}'_{0}(\tilde{v}), \tilde{v}\rangle=0.
\end{equation*}
Now, we show that $\{y_{n}\}_{n\in \mathbb{N}}$ has a subsequence such that $y_{n}\rightarrow y_{0}\in \Lambda$. 
Assume by contradiction that $\{y_{n}\}_{n\in \mathbb{N}}$ is not bounded, that is there exists a subsequence, still denoted by $\{y_{n}\}_{n\in \mathbb{N}}$, such that $|y_{n}|\rightarrow +\infty$. 
Since $u_{n}\in \N_{\e_{n}}$, we can see that
\begin{align*}
&m_{0}\|\tilde{u}_{n}\|_{0}^{2}\leq \int_{\R^{3}} g(\e_{n} x+y_{n}, \tilde{u}_{n})\tilde{u}_{n}\, dx.
\end{align*}
Take $R>0$ such that $\Omega \subset B_{R}(0)$. We may assume that $|y_{n}|>2R$, so, for any $x\in B_{R/\e_{n}}(0)$ we get $|\e_{n} x+y_{n}|\geq |y_{n}|-|\e_{n} x|>R$.\\
Then, we deduce that
\begin{align*}
&m_{0}\|v_{n}\|_{0}^{2}\leq \int_{B_{R/\e_{n}}(0)}  \tilde{f}(\tilde{u}_{n}) \tilde{u}_{n} \,dx+\int_{\R^{3}\setminus B_{R/\e_{n}}(0)} f(\tilde{u}_{n}) \tilde{u}_{n} \, dx.
\end{align*}
Since $\tilde{u}_{n}\rightarrow \tilde{u}$ in $H^{s}(\R^{3})$, from the Dominated Convergence Theorem we can see that 
\begin{align*}
\int_{\R^{3}\setminus B_{R/\e_{n}}(0)} f(\tilde{u}_{n}) \tilde{u}_{n} \, dx=o_{n}(1).
\end{align*}
Recalling that $\tilde{f}(\tilde{u}_{n})\leq \frac{V_{0}}{K} \tilde{u}_{n}$, we get
\begin{align*}
m_{0}\|\tilde{u}_{n}\|^{2}_{0}\leq \frac{1}{K} \int_{B_{R/\e_{n}}(0)}  V_{0} \tilde{u}^{2}_{n} \,dx+o_{n}(1),
\end{align*}
which yields
$$
\left(m_{0}-\frac{1}{K}\right)\|\tilde{u}_{n}\|^{2}_{0}\leq o_{n}(1).
$$
Since $\tilde{u}_{n}\rightarrow \tilde{u}\neq 0$, we have a contradiction.
Thus $\{y_{n}\}_{n\in \mathbb{N}}$ is bounded and, up to a subsequence, we may assume that $y_{n}\rightarrow y_{0}$. If $y_{0}\notin \overline{\Omega}$, then there exists $r>0$ such that $y_{n}\in B_{r/2}(y_{0})\subset \R^{3}\setminus \overline{\Omega}$ for any $n$ large enough. Reasoning as before, we get a contradiction. Hence $y\in \overline{\Omega}$. \\
Now, we prove that $V(y_{0})=V_{0}$. Assume by contradiction that $V(y_{0})>V_{0}$.
Taking into account \eqref{3.22}, Fatou's Lemma and the invariance of $\R^{3}$ by translations, we have
\begin{align*}
c_{0}&< \liminf_{n\rightarrow \infty} \Bigl[ \frac{1}{2} \widehat{M}\left(\int_{\R^{3}} |(-\Delta)^{\frac{s}{2}} \tilde{v}_{n}|^{2}+V(\e_{n} z+y_{n}) \tilde{v}^{2}_{n} \right)- \int_{\R^{3}} F(\tilde{v}_{n})\, dx \Bigr] \nonumber\\
&\leq \liminf_{n\rightarrow \infty} \mathcal{J}_{\e_{n}}(t_{n}u_{n}) \leq \liminf_{n\rightarrow \infty} \mathcal{J}_{\e_{n}} (u_{n})=c_{0}
\end{align*}
which gives a contradiction. 
\end{proof}

\noindent
Now, we aim to relate the number of positive solutions of \eqref{Pea} to the topology of the set $\Lambda$.
For this reason, we take $\delta>0$ such that
$$
\Lambda_{\delta}=\{x\in \R^{3}: {\rm dist}(x, \Lambda)\leq \delta\}\subset \Omega,
$$
and we consider $\eta\in C^{\infty}_{0}(\R_{+}, [0, 1])$ such that $\eta(t)=1$ if $0\leq t\leq \frac{\delta}{2}$ and $\eta(t)=0$ if $t\geq \delta$.\\
For any $y\in \Lambda$, we define 
$$
\Psi_{\e, y}(x)=\eta(|\e x-y|) w\left(\frac{\e x-y}{\e}\right) 
$$
where $w\in H^{s}(\R^{3})$ is a positive ground state solution to the autonomous problem \eqref{P0} (such a solution exists in view of Theorem \ref{thm3.1}).\\
Let $t_{\e}>0$ be the unique number such that 
$$
\max_{t\geq 0} \mathcal{J}_{\e}(t \Psi_{\e, y})=\mathcal{J}_{\e}(t_{\e} \Psi_{\e, y}). 
$$
Finally, we consider $\Phi_{\e}: \Lambda\rightarrow \N_{\e}$ defined by setting
$$
\Phi_{\e}(y)= t_{\e} \Psi_{\e, y}.
$$
\begin{lem}\label{lem3.4}
The functional $\Phi_{\e}$ satisfies the following limit
\begin{equation*}
\lim_{\e\rightarrow 0} \mathcal{J}_{\e}(\Phi_{\e}(y))=c_{0} \mbox{ uniformly in } y\in \Lambda.
\end{equation*}
\end{lem}
\begin{proof}
Assume by contradiction that there exists $\delta_{0}>0$, $\{y_{n}\}_{n\in \mathbb{N}}\subset \Lambda$ and $\e_{n}\rightarrow 0$ such that 
\begin{equation}\label{puac}
|\mathcal{J}_{\e_{n}}(\Phi_{\e_{n}}(y_{n}))-c_{0}|\geq \delta_{0}.
\end{equation}
Let us observe that by using the change of variable $\displaystyle{z=\frac{\e_{n}x-y_{n}}{\e_{n}}}$, if $z\in B_{\frac{\delta}{\e_{n}}}(0)\subset \Lambda_{\delta}\subset \Omega$, it follows that $\e_{n} z\in B_{\delta}(0)$ and $\e_{n} x+y_{n}\in B_{\delta}(y_{n})\subset \Lambda_{\delta}$. \\
Then, recalling that $G=F$ in $\Omega$ and $\eta(t)=0$ for $t\geq \delta$, we have
\begin{align}\label{HZ}
\mathcal{J}_{\e}(\Phi_{\e_{n}}(y_{n}))&=\frac{1}{2} \widehat{M}\left( t^{2}_{\e_{n}} \left(\int_{\R^{3}} |(-\Delta)^{\frac{s}{2}} (\eta(|\e_{n} z|) w(z))|^{2} \, dz+\int_{\R^{3}} V(\e_{n} z+y_{n}) (\eta(|\e_{n} z|) w(z))^{2} \, dz \right)\right) \nonumber\\
&-\int_{\R^{3}} F(t_{\e_{n}} \eta(|\e_{n} z|) w(z)) \, dz.
\end{align}
Now, we aim to show that the sequence $\{t_{\e_{n}}\}_{n\in \mathbb{N}}$ verifies $t_{\e_{n}}\rightarrow 1$ as $\e_{n}\rightarrow 0$.
From the definition of $t_{\e_{n}}$, it follows that $\langle \mathcal{J}'_{\e_{n}}(\Phi_{\e_{n}}(y_{n})),\Phi_{\e_{n}}(y_{n})\rangle=0 $, which gives
\begin{align}\label{nio}
\frac{M(t^{2}_{\e_{n}}A_{n}^{2})}{t^{2}_{\e_{n}} A_{n}^{2}}=\frac{1}{A_{n}^{4}} \int_{\R^{3}} \Bigl[ \frac{f(t_{\e_{n}} \eta(|\e_{n} z|) w(z))}{(t_{\e_{n}} \eta(|\e_{n} z|) w(z))^{3}}\Bigr] (\eta(|\e_{n} z|) w(z))^{4} \, dz
\end{align}
where 
$$
A_{n}=\int_{\R^{3}} |(-\Delta)^{\frac{s}{2}} (\eta(|\e_{n} z|) w(z))|^{2} \, dz+\int_{\R^{3}} V(\e_{n} z+y_{n}) (\eta(|\e_{n} z|) w(z))^{2} \, dz.
$$
Since $\eta=1$ in $B_{\frac{\delta}{2}}(0)\subset B_{\frac{\delta}{\e_{n}}}(0)$ for all $n$ large enough, we get from \eqref{nio}
\begin{align*}
\frac{M(t^{2}_{\e_{n}}A_{n}^{2})}{t^{2}_{\e_{n}} A_{n}^{2}}\geq \frac{1}{A_{n}^{4}} \int_{B_{\frac{\delta}{2}}(0)} \Bigl[ \frac{f(t_{\e_{n}} w(z))}{(t_{\e_{n}} w(z))^{3}}\Bigr] w^{4}(z) \, dz.
\end{align*}
From the continuity of $w$ we can find a vector $\hat{z}\in \R^{3}$ such that 
\begin{equation*}
w(\hat{z})=\min_{z\in B_{\frac{\delta}{2}}} w(z)>0,
\end{equation*} 
so, by using $(f_4)$, we deduce that
\begin{align}\label{nioo}
\frac{M(t^{2}_{\e_{n}}A_{n}^{2})}{t^{2}_{\e_{n}} A_{n}^{2}}\geq \frac{1}{A_{n}^{4}} \Bigl[ \frac{f(t_{\e_{n}} w(\hat{z}))}{(t_{\e_{n}} w(\hat{z}))^{3}}\Bigr] \int_{B_{\frac{\delta}{2}}(0)} w^{4}(z) \, dz\geq \frac{1}{A_{n}^{4}} \Bigl[ \frac{f(t_{\e_{n}} w(\hat{z}))}{(t_{\e_{n}} w(\hat{z}))^{3}}\Bigr] w^{4}(\hat{z}) |B_{\frac{\delta}{2}}(0)|.
\end{align}
Now, assume by contradiction that $t_{\e_{n}}\rightarrow \infty$. 
Let us observe that Lemma \ref{pp} yields  
\begin{align}\label{nio3}
\| \Psi_{\e_{n}, y_{n}} \|^{2}_{\e_{n}}=A_{n}^{2}\rightarrow \|w\|^{2}_{0}\in (0, \infty).
\end{align}
So, by using $t_{\e_{n}}\rightarrow \infty$, $(M_3)$ and \eqref{nio3}, we can see that
\begin{align}\label{nio1}
\lim_{n\rightarrow \infty} \frac{M(t^{2}_{\e_{n}}A_{n}^{2})}{t^{2}_{\e_{n}} A_{n}^{2}}\leq \lim_{n\rightarrow \infty} \frac{\gamma(1+ t^{2}_{\e_{n}}A_{n}^{2})}{t^{2}_{\e_{n}} A_{n}^{2}}\leq C<\infty.
\end{align}
On the other hand, the assumption $(f_3)$ implies that
\begin{align}\label{nio2}
\lim_{n\rightarrow \infty} \frac{f(t_{\e_{n}} w(\hat{z}))}{(t_{\e_{n}} w(\hat{z}))^{3}}=\infty.
\end{align}
Putting together \eqref{nioo}, \eqref{nio1} and \eqref{nio2} we have a contradiction.
Therefore $\{t_{\e_{n}}\}_{n\in \mathbb{N}}$ is bounded and, up to subsequence, we may assume that $t_{\e_{n}}\rightarrow t_{0}$ for some $t_{0}\geq 0$. 
Let us prove that $t_{0}>0$. Suppose by contradiction that $t_{0}=0$. 
Then, taking into account \eqref{nio3} and the assumptions $(M_1)$, $(f_1)$ and $(f_2)$, we can see that \eqref{nio} yields
\begin{align*}
0<m_{0}\|w\|^{2}_{0}\leq \lim_{n\rightarrow \infty} \left[C t^{2}_{\e_{n}} \int_{\R^{3}} w^{4}\, dz+C t_{\e_{n}}^{q-2} \int_{\R^{3}} w^{q} \, dz \right]=0
\end{align*}
which is impossible. Hence $t_{0}>0$.
Thus, by passing to the limit as $n\rightarrow \infty$ in \eqref{nio}, we deduce from \eqref{nio3}, the continuity of $M$ and the Dominated Convergence Theorem that
\begin{align*}
M(t_{0}^{2} \|w\|^{2}_{0})\, t_{0} \|w\|^{2}_{0}=\int_{\R^{3}} f(t_{0} w) \,w \, dx.
\end{align*}
Since $w\in \N_{0}$, we can see that 
\begin{align}\label{buondi}
\frac{M(t_{0}^{2} \|w\|^{2}_{0})}{t_{0}^{2} \|w\|^{2}_{0}}-\frac{M(\|w\|^{2}_{0})}{\|w\|^{2}_{0}}=\frac{1}{\|w\|^{4}_{0}}\int_{\R^{3}} \left[\frac{f(t_{0} w)}{(t_{0}w)^{3}}- \frac{f(w)}{w^{3}} \right] w^{4} \, dx.
\end{align}
If $t_{0}>1$, from $(M_3)$ and $(f_4)$ we can see that the left hand side of \eqref{buondi} is negative and the right hand side is positive. A similar reasoning can be done when $t_{0}<1$. Therefore $t_{0}=1$.\\
Then, taking the limit as $n\rightarrow \infty$ in \eqref{HZ} and by using $t_{\e_{n}}\rightarrow 1$,
\begin{align*}
\int_{\R^{3}} F(\eta(|\e_{n} z|) w(z))  \, dz \rightarrow \int_{\R^{3}} F(w)  \, dz 
\end{align*}
(this follows by the Dominated Convergence Theorem) and \eqref{nio3}, we obtain
$$
\lim_{n\rightarrow \infty} \mathcal{J}_{\e_{n}}(\Phi_{\e_{n}, y_{n}})=\mathcal{J}_{0}(w)=c_{0},
$$
which contradicts \eqref{puac}.
\end{proof}

\noindent
At this point, we are in the position to define the barycenter map. For any $\delta>0$, we take $\rho=\rho(\delta)>0$ such that $\Lambda_{\delta}\subset B_{\rho}$, and we consider $\varUpsilon: \R^{3}\rightarrow \R^{3}$ defined by setting
\begin{equation*}
\varUpsilon(x)=
\left\{
\begin{array}{ll}
x &\mbox{ if } |x|<\rho \\
\frac{\rho x}{|x|} &\mbox{ if } |x|\geq \rho.
\end{array}
\right.
\end{equation*}
We define the barycenter map $\beta_{\e}: \N_{\e}\rightarrow \R^{N}$ as follows
\begin{align*}
\beta_{\e}(u)=\frac{\displaystyle{\int_{\R^{3}} \varUpsilon(\e x)u^{2}(x) \,dx}}{\displaystyle{\int_{\R^{3}} u^{2}(x) \,dx}}.
\end{align*}

\begin{lem}\label{lem3.5N}
The function $\beta_{\e}$ verifies the following limit
\begin{equation*}
\lim_{\e \rightarrow 0} \beta_{\e}(\Phi_{\e}(y))=y \mbox{ uniformly in } y\in \Lambda.
\end{equation*}
\end{lem}
\begin{proof}
Assume by contradiction that there exists $\delta_{0}>0$, $\{y_{n}\}_{n\in \mathbb{N}}\subset \Lambda$ and $\e_{n}\rightarrow 0$ such that 
\begin{equation}\label{4.4}
|\beta_{\e_{n}}(\Phi_{\e_{n}}(y_{n}))-y_{n}|\geq \delta_{0}.
\end{equation}
From the definitions of $\Phi_{\e_{n}}(y_{n})$, $\beta_{\e_{n}}$, $\eta$ and by using the change of variable $\displaystyle{z= \frac{\e_{n} x-y_{n}}{\e_{n}}}$
we can see that 
$$
\beta_{\e_{n}}(\Phi_{\e_{n}}(y_{n}))=y_{n}+\frac{\displaystyle{\int_{\R^{3}}[\Upsilon(\e_{n}z+y_{n})-y_{n}] |\eta(|\e_{n}z|)|^{2} |w(z)|^{2} \, dz}}{\displaystyle{\int_{\R^{3}} |\eta(|\e_{n}z|)|^{2} |w(z)|^{2}\, dz}}.
$$
Since $\{y_{n}\}_{n\in \mathbb{N}}\subset \Lambda\subset B_{\rho}(0)$ and by applying the Dominated Convergence Theorem, we can deduce
$$
|\beta_{\e_{n}}(\Phi_{\e_{n}}(y_{n}))-y_{n}|=o_{n}(1)
$$
which is in contrast with (\ref{4.4}).
\end{proof}

\noindent
At this point, we introduce a subset $\widetilde{\N}_{\e}$ of $\N_{\e}$ by taking a function $h_{1}:\R^{+}\rightarrow \R^{+}$ such that $h_{1}(\e)\rightarrow 0$ as $\e \rightarrow 0$, and setting
$$
\widetilde{\N}_{\e}=\left \{u\in \N_{\e}: \mathcal{J}_{\e}(u)\leq c_{0}+h_{1}(\e)\right\}.
$$
Fixed $y\in \Lambda$, from Lemma \ref{lem3.4} follows that $h_{1}(\e)=|\mathcal{J}_{\e}(\Phi_{\e}(y))-c_{0}|\rightarrow 0$ as $\e \rightarrow 0$. Therefore $\Phi_{\e}(y)\in \widetilde{\N}_{\e}$, and $\widetilde{\N}_{\e}\neq \emptyset$ for any $\e>0$. Moreover, we have the following lemma.
\begin{lem}\label{lem3.5}
$$
\lim_{\e \rightarrow 0} \sup_{u\in \widetilde{\mathcal{N}}_{\e}} {\rm dist}(\beta_{\e}(u), \Lambda_{\delta})=0.
$$
\end{lem}

\begin{proof}
Let $\e_{n}\rightarrow 0$ as $n\rightarrow \infty$. For any $n\in \mathbb{N}$, there exists $u_{n}\in \widetilde{\N}_{\e_{n}}$ such that
$$
\sup_{u\in \widetilde{\N}_{\e_{n}}} \inf_{y\in \Lambda_{\delta}}|\beta_{\e_{n}}(u)-y|=\inf_{y\in \Lambda_{\delta}}|\beta_{\e_{n}}(u_{n})-y|+o_{n}(1).
$$
For this reason, it is enough to prove that there exists $\{y_{n}\}_{n\in \mathbb{N}}\subset \Lambda_{\delta}$ such that 
\begin{equation}\label{3.13}
\lim_{n\rightarrow \infty} |\beta_{\e_{n}}(u_{n})-y_{n}|=0.
\end{equation}
Since $\mathcal{J}_{0}(tu_{n})\leq \mathcal{J}_{\e_{n}}(tu_{n})$ for all $t\geq 0$ and $\{u_{n}\}_{n\in \mathbb{N}}\subset  \widetilde{\N}_{\e_{n}}\subset \N_{\e_{n}}$, we deduce that
$$
c_{0}\leq c_{\e_{n}}\leq \mathcal{J}_{\e_{n}}(u_{n})\leq c_{0}+h_{1}(\e_{n}),
$$
and this implies that $\mathcal{J}_{\e_{n}}(u_{n})\rightarrow c_{0}$. By using Lemma \ref{lem3.1N}, there exists $\{\tilde{y}_{n}\}_{n\in \mathbb{N}}\subset \R^{3}$ such that $y_{n}=\e_{n}\tilde{y}_{n}\in \Lambda_{\delta}$ for $n$ sufficiently large. By setting $\tilde{u}_{n}(x)=u_{n}(\cdot+\tilde{y}_{n})$, we can see that
$$
\beta_{\e_{n}}(u_{n})=y_{n}+\frac{\displaystyle{\int_{\R^{3}}[\Upsilon(\e_{n}x+y_{n})-y_{n}] \tilde{u}_{n}^{2} \, dx}}{\displaystyle{\int_{\R^{3}} \tilde{u}_{n}^{2} \, dx}}=y_{n}+o_{n}(1),
$$
because $\tilde{u}_{n}\rightarrow u$ in $H^{s}(\R^{3})$ and $\e_{n} x+y_{n}\rightarrow y \in\Lambda$. As a consequence, the sequence $\{\tilde{y}_{n}\}_{n\in \mathbb{N}}$ verifies (\ref{3.13}).
\end{proof}

\noindent
Before proving our multiplicity result for the modified problem \eqref{Pea}, we recall the following useful abstract result whose proof can be found in \cite{BC}.
\begin{lem}\label{BC}
Let $I$, $I_{1}$ and $I_{2}$ be closed sets with $I_{1}\subset I_{2}$, and let $\pi: I\rightarrow I_{2}$ and $\psi: I_{1}\rightarrow I$ be two continuous maps  such that $\pi \circ \psi$ is homotopically equivalent to the embedding $j: I_{1}\rightarrow I_{2}$. Then $cat_{I}(I)\geq cat_{I_{2}}(I_{1})$.
\end{lem}

\begin{thm}\label{thm3.2}
Assume that $(M_1)$-$(M_3)$, $(V_1)$-$(V_2)$ and $(f_1)$-$(f_4)$ hold true. 
Then, given $\delta>0$ there exists $\bar{\e}_\delta>0$ such that, for any $\e \in (0, \bar{\e}_\delta)$, problem $\eqref{Pea}$ has at least $cat_{\Lambda_{\delta}}(\Lambda)$ positive solutions.   
\end{thm}

\begin{proof}
For any $\e>0$, we consider the map $\alpha_{\e} : \Lambda \rightarrow \mathbb{S}_{\e}^{+}$ defined as $\alpha_{\e}(y)= m_{\e}^{-1}(\Phi_{\e}(y))$. \\
By using Lemma \ref{lem3.4}, we can see that
\begin{equation}\label{FJS}
\lim_{\e \rightarrow 0} \psi_{\e}(\alpha_{\e}(y)) = \lim_{\e \rightarrow 0} \mathcal{J}_{\e}(\Phi_{\e}(y))= c_{0} \mbox{ uniformly in } y\in \Lambda. 
\end{equation}  
Set
$$
\widetilde{\mathcal{S}}^{+}_{\e}=\{ w\in \mathbb{S}_{\e}^{+} : \psi_{\e}(w) \leq c_{0} + h_{1}(\e)\}, 
$$
where $h_{1}(\e)\rightarrow 0$ as $\e\rightarrow 0^{+}$. It follows from \eqref{FJS} that $h_{1}(\e)=|\psi_{\e}(\alpha_{\e}(y))-c_{0}|\rightarrow 0$ as $\e\rightarrow 0^{+}$ uniformly in $y\in \Lambda$, so 
there exists $\bar{\e}>0$ such that $\psi_{\e}(\alpha_{\e}(y))\in \widetilde{\mathcal{S}}^{+}_{\e}$  and $\widetilde{\mathcal{S}}^{+}_{\e}\neq \emptyset$ for all $\e \in (0, \bar{\e})$.
From Lemma \ref{lem3.4}, Lemma \ref{lem2.3}-$(iii)$, Lemma \ref{lem3.5} and Lemma \ref{lem3.5N}, we can find $\bar{\e}= \bar{\e}_{\delta}>0$ such that the following diagram
\begin{equation*}
\Lambda\stackrel{\Phi_{\e}}{\rightarrow} \Phi_{\e}(\Lambda) \stackrel{m_{\e}^{-1}}{\rightarrow} \alpha_{\e}(\Lambda)\stackrel{m_{\e}}{\rightarrow} \Phi_{\e}(\Lambda) \stackrel{\beta_{\e}}{\rightarrow} \Lambda_{\delta}
\end{equation*}    
is well defined for any $\e \in (0, \bar{\e})$. \\
Thanks to Lemma \ref{lem3.5N}, and decreasing $\bar{\e}$ if necessary, we can see that $\beta_{\e}(\Phi_{\e}(y))= y+ \theta(\e, y)$ for all $y\in \Lambda$, for some function $\theta(\e, y)$ verifying $|\theta(\e, y)|<\frac{\delta}{2}$ uniformly in $y\in \Lambda$ and for all $\e \in (0, \bar{\e})$. Then, we can see that $H(t, y)= y+ (1-t)\theta(\e, y)$ with $(t, y)\in [0,1]\times \Lambda$ is a homotopy between $\beta_{\e} \circ \Phi_{\e} = (\beta_{\e} \circ m_{\e}) \circ (m_{\e}^{-1}\circ \Phi_{\e})$ and the inclusion map $id: \Lambda \rightarrow \Lambda_{\delta}$. This fact together with Lemma \ref{BC} implies that 
\begin{equation}\label{cat}
cat_{\alpha_{\e}(\Lambda)} \alpha_{\e}(\Lambda)\geq cat_{\Lambda_{\delta}}(\Lambda).
\end{equation}
Therefore, by using Corollary \ref{cor2.1} and Corollary $28$ in \cite{SW}, with $c= c_{\e}\leq c_{0}+h_{1}(\e) =d$ and $K= \alpha_{\e}(\Lambda)$, we can see that $\Psi_{\e}$ has at least $cat_{\alpha_{\e}(\Lambda)} \alpha_{\e}(\Lambda)$ critical points on $\widetilde{\mathcal{S}}^{+}_{\e}$.
Taking into account Proposition \ref{prop2.1}-$(d)$ and \eqref{cat}, we can infer that $\mathcal{J}_{\e}$ admits at least $cat_{\Lambda_{\delta}}(\Lambda)$ critical points in $\widetilde{\mathcal{N}}_{\e}$.       
\end{proof}

\section{proof of theorem \ref{thm1}}

\noindent
In this last section, we provide the proof of Theorem \ref{thm1}. Firstly, we establish the following useful $L^{\infty}$-estimate for the solutions of the modified problem \eqref{Pea}. The proof is obtained by adapting in nonlocal setting the Moser iteration technique \cite{Moser}.
\begin{lem}\label{Moser}
Let $\e_{n}\rightarrow 0$ and $u_{n}\in \widetilde{\N}_{\e_{n}}$ be a solution to \eqref{Pea}. Then, up to a subsequence, $\tilde{u}_{n}=u_{n}(\cdot+\tilde{y}_{n})\in L^{\infty}(\R^{N})$, and there exists $C>0$ such that 
\begin{equation*}
\|\tilde{u}_{n}\|_{L^{\infty}(\R^{3})}\leq C \mbox{ for all } n\in \mathbb{N}.
\end{equation*}
\end{lem}
\begin{proof}
Firstly, we note that $\tilde{u}_{n}$ is a sub-solution to the following equation
\begin{equation}\label{subsol}
(-\Delta)^{s} \tilde{u}_{n}+V(\e_{n} x+\e_{n} \tilde{y}_{n}) \tilde{u}_{n}=\frac{1}{m_{0}} g(\e_{n} x+\e_{n} \tilde{y}_{n}, \tilde{u}_{n}) \mbox{ in } \R^{3}.
\end{equation}
For any $n\in \mathbb{N}$ and $L>0$, we define $v_{n, L}=\tilde{u}_{n} \tilde{u}_{n, L}^{2(\beta-1)}$ where $\tilde{u}_{n, L}=\min\{\tilde{u}_{n}, L\}$ and $\beta>1$ will be determined later. \\
Let $\phi(t)=\phi_{L, \beta}(t)=t t_{L}^{2(\beta-1)}$ and we observe that 
\begin{equation}\label{Gg}
h'(a-b)(\phi(a)-\phi(b))\geq |\Phi(a)-\Phi(b)|^{2} \mbox{ for any } a, b\in\R,
\end{equation}
where 
$$
h(t)= \frac{|t|^{2}}{2} \, \mbox{ and } \, \Phi(t)=\int_{0}^{t} (\phi'(\tau))^{\frac{1}{2}} d\tau. 
$$
Indeed, since $\phi$ is an increasing function, we can see that
$$
(a-b)(\phi(a)-\phi(b))\geq 0 \mbox{ for any } a, b\in \R.
$$ 
Now, fix $a, b\in \R$ such that $a>b$. From the definition of $\Phi$ and Jensen inequality, we get
\begin{align*}
h'(a-b)(\phi(a)-\phi(b))&=(a-b) (\phi(a)-\phi(b)) \\
&= (a-b) \int_{a}^{b} \phi'(t) dt \\
&= (a-b) \int_{a}^{b} (\Phi'(t))^{2} dt \\
&\geq \left(\int_{a}^{b} (\Phi'(t)) dt\right)^{2}.
\end{align*}
In similar fashion, we can prove that the above inequality is true for any $a\leq b$, so \eqref{Gg} holds. 
Taking $v_{n, L}(\geq 0)$ as test-function in the weak formulation of \eqref{subsol} and by using \eqref{Gg}, we can see that 
\begin{align}\label{BMS}
[\Gamma(\tilde{u}_{n})]^{2}
&\leq \iint_{\R^{6}} \frac{(\tilde{u}_{n}(x)- \tilde{u}_{n}(y))}{|x-y|^{3+2s}} ((\tilde{u}_{n} \tilde{u}_{n,L}^{2(\beta-1)})(x)-(\tilde{u}_{n} \tilde{u}_{n,L}^{2(\beta-1)})(y)) \,dx dy \nonumber\\
&\leq \frac{1}{m_{0}}\int_{\R^{3}} [g(\e_{n} x+\e_{n}\tilde{y}_{n}, \tilde{u}_{n})-m_{0} V(\e_{n} x+\e_{n} \tilde{y}_{n})]  \tilde{u}_{n} \tilde{u}_{n, L}^{2(\beta-1)} \,dx.
\end{align}
Since 
$$
\Gamma(\tilde{u}_{n})\geq \frac{1}{\beta} \tilde{u}_{n} \tilde{u}_{n, L}^{\beta-1},
$$
from the Sobolev inequality in Theorem \ref{Sembedding} we can deduce that 
$$
[\Gamma(\tilde{u}_{n})]^{2}\geq S^{-1}_{*} \|\Gamma(\tilde{u}_{n})\|^{2}_{L^{2^{*}_{s}}(\R^{3})}\geq \left(\frac{1}{\beta}\right)^{2} S^{-1}_{*}\|\tilde{u}_{n} \tilde{u}_{n, L}^{\beta-1}\|^{2}_{L^{2^{*}_{s}}(\R^{3})}.
$$
This together with \eqref{BMS}, $(M_1)$, $(g_1)$ and $(g_2)$ implies that
\begin{align*}
\left(\frac{1}{\beta}\right)^{2} S^{-1}_{*}\|\tilde{u}_{n} \tilde{u}_{n, L}^{\beta-1}\|^{2}_{L^{\bar{2}^{*}_{s}}(\R^{3})}&\leq \frac{1}{m_{0}}\int_{\R^{3}}   [g(\e_{n} x+\e_{n}\tilde{y}_{n}, \tilde{u}_{n}) -m_{0} V(\e_{n} x+\e_{n} \tilde{y}_{n})]  \tilde{u}_{n} \tilde{u}_{n, L}^{2(\beta-1)} \,dx \\
&\leq C \int_{\R^{3}}  \tilde{u}_{n}^{q} \tilde{u}_{n, L}^{2(\beta-1)} dx \\
&= C \int_{\R^{3}}  \tilde{u}_{n}^{q-2} (\tilde{u}_{n} \tilde{u}_{n, L}^{\beta-1})^{2} dx. 
\end{align*}
Now, we set $w_{n, L}:=\tilde{u}_{n} \tilde{u}_{n, L}^{\beta-1}$. Then, by using H\"older inequality, we deduce that
\begin{align*}
\|w_{n, L}\|_{L^{2^{*}_{s}}(\R^{3})}^{2}\leq C \beta^{2}  \left(\int_{\R^{3}} \tilde{u}_{n}^{2^{*}_{s}}\, dx\right)^{\frac{q-2}{2^{*}_{s}}} \left(\int_{\R^{3}} w_{n, L}^{\frac{2 2^{*}_{s}}{2^{*}_{s}-(q-2)}} \, dx\right)^{\frac{2^{*}_{s}-(q-2)}{2^{*}_{s}}}
\end{align*}
where $2<\frac{2 2^{*}_{s}}{2^{*}_{s}-(q-2)}<2^{*}_{s}$. 
Recalling that $\{\tilde{u}_{n}\}_{n\in \mathbb{N}}$ is bounded in $\mathcal{H}_{\e_{n}}$, we can see that
\begin{align}\label{conto4}
\|w_{n, L}\|_{L^{2^{*}_{s}}(\R^{3})}^{2}\leq C \beta^{2}  \|w_{n, L}\|_{L^{\alpha^{*}_{s}}(\R^{3})}^{2}
\end{align}
where 
$$
\alpha^{*}_{s}=\frac{2 2^{*}_{s}}{2^{*}_{s}-(q-2)}.
$$ 
We observe that if $\tilde{u}_{n}^{\beta}\in L^{\alpha^{*}_{s}}(\R^{3})$, from the definition of $w_{n, L}$, the fact that $\tilde{u}_{n, L}\leq \tilde{u}_{n}$, and  (\ref{conto4}), we deduce
\begin{align}\label{conto5}
\|w_{n, L}\|_{L^{2^{*}_{s}}(\R^{3})}^{2}\leq C \beta^{2} \left(\int_{\R^{3}} \tilde{u}_{n}^{\beta \alpha^{*}}\, dx\right)^{\frac{2}{\alpha^{*}_{s}}}<\infty.
\end{align}
By passing to the limit in (\ref{conto5}) as $L \rightarrow +\infty$, the Fatou's Lemma yields
\begin{align}\label{conto6}
\|\tilde{u}_{n}\|_{L^{\beta 2^{*}_{s}}(\R^{3})}\leq C^{\frac{1}{2\beta}} \beta^{\frac{1}{\beta}} \|\tilde{u}_{n}\|_{L^{\beta \alpha^{*}_{s}}(\R^{3})}
\end{align}
whenever $\tilde{u}_{n}^{\beta \alpha^{*}}\in L^{1}(\R^{3})$.\\
Now, we set $\beta:=\frac{2^{*}_{s}}{\alpha^{*}_{s}}>1$. Since $\tilde{u}_{n}\in L^{2^{*}_{s}}(\R^{3})$, the above inequality holds for this choice of $\beta$. Then, by using the fact that $\beta^{2}\alpha^{*}_{s}=\beta \,2^{*}_{s}$, it follows that \eqref{conto6} holds with $\beta$ replaced by $\beta^{2}$.
Therefore, we can see that
\begin{align*}
\|\tilde{u}_{n}\|_{L^{\beta^{2} 2^{*}_{s}}(\R^{3})}&\leq C^{\frac{1}{2\beta^{2}}} \beta^{\frac{2}{\beta^{2}}} \|\tilde{u}_{n}\|_{L^{\beta^{2} \alpha^{*}_{s}}(\R^{3})}\\
&\leq  C^{\frac{1}{2}\left(\frac{1}{\beta}+\frac{1}{\beta^{2}}\right)} \beta^{\frac{1}{\beta}+\frac{2}{\beta^{2}}} \|\tilde{u}_{n}\|_{L^{\beta \alpha^{*}}(\R^{3})}.
\end{align*}
Iterating this process, and recalling that $\beta\, \alpha^{*}:=2^{*}_{s}$, we can infer that for every $m\in \mathbb{N}$
\begin{align}\label{conto7}
\|\tilde{u}_{n}\|_{L^{\beta^{m} 2^{*}_{s}}(\R^{3})}\leq C^{\sum_{j=1}^{m}\frac{1}{2\beta^{j}}} \beta^{\sum_{j=1}^{m} j\beta^{-j}} \|\tilde{u}_{n}\|_{L^{2^{*}_{s}}(\R^{3})}.
\end{align}
Taking the limit in (\ref{conto7}) as $m \rightarrow +\infty$, we get
\begin{align*}
\|\tilde{u}_{n}\|_{L^{\infty}(\R^{3})}\leq C \|\tilde{u}_{n}\|_{L^{2^{*}_{s}}(\R^{3})}\leq C \mbox{ for any } n\in \mathbb{N}.
\end{align*}
\end{proof}

\noindent
Now, we are able to give the proof of our main result.
\begin{proof}[Proof of Theorem \ref{thm1}]
Take $\delta>0$ such that $\Lambda_\delta \subset \Omega$. We begin proving that there exists $\tilde{\e}_{\delta}>0$ such that for any $\e \in (0, \tilde{\e}_{\delta})$ and any solution $u_{\e} \in \widetilde{\mathcal{N}}_{\e}$ of \eqref{Pea}, it results 
\begin{equation}\label{infty}
\|u_{\e}\|_{L^{\infty}(\R^{3}\setminus \Omega_{\e})}<a. 
\end{equation}
Suppose by contradiction that for some subsequence $\{\e_{n}\}_{n\in \mathbb{N}}$ such that $\e_{n}\rightarrow 0$, we can find $u_{\e_{n}}\in \widetilde{\mathcal{N}}_{\e_{n}}$ such that $\mathcal{J}'_{\e_{n}}(u_{\e_{n}})=0$ and 
\begin{equation}\label{eee}
\|u_{\e_{n}}\|_{L^{\infty}(\R^{3}\setminus \Omega_{\e_{n}})}\geq a.
\end{equation} 
Since $\mathcal{J}_{\e_{n}}(u_{\e_{n}}) \leq c_{0} + h_{1}(\e_{n})$ and $h_{1}(\e_{n})\rightarrow 0$, we can proceed as in the first part of the proof of Lemma \ref{lem3.1N}, to deduce that $\mathcal{J}_{\e_{n}}(u_{\e_{n}})\rightarrow c_{0}$.
Then, by using Lemma \ref{lem3.1N}, we can find $\{\tilde{y}_{n}\}_{n\in \mathbb{N}}\subset \R^{3}$ such that $\tilde{u}_{n}=u_{\e_{n}}(\cdot+\tilde{y}_{n})\rightarrow \tilde{u}$ in $H^{s}(\R^{3})$ and $\e_{n}\tilde{y}_{n}\rightarrow y_{0} \in \Lambda$. \\
Now, if we choose $r>0$ such that $B_{r}(y_{0})\subset B_{2r}(y_{0})\subset \Omega$, we can see that $B_{\frac{r}{\e_{n}}}(\frac{y_{0}}{\e_{n}})\subset \Omega_{\e_{n}}$. In particular, for any $y\in B_{\frac{r}{\e_{n}}}(\tilde{y}_{n})$ it holds
\begin{equation*}
\left|y - \frac{y_{0}}{\e_{n}}\right| \leq |y- \tilde{y}_{n}|+ \left|\tilde{y}_{n} - \frac{y_{0}}{\e_{n}}\right|<\frac{1}{\e_{n}}(r+o_{n}(1))<\frac{2r}{\e_{n}}\, \mbox{ for } n \mbox{ sufficiently large. }
\end{equation*}
Therefore 
\begin{equation}\label{ern}
\R^{3}\setminus \Omega_{\e_{n}}\subset \R^{3} \setminus B_{\frac{r}{\e_{n}}}(\tilde{y}_{n})
\end{equation}
for any $n$ big enough.\\ 
Now, we observe that $\tilde{u}_{n}$ is a solution to
\begin{align*}
(-\Delta)^{s} \tilde{u}_{n}+\tilde{u}_{n}=\xi_{n} \mbox{ in } \R^{3},
\end{align*}
where 
$$
\xi_{n}(x):=\frac{1}{M\left(\int_{\R^{3}} |(-\Delta)^{\frac{s}{2}}\tilde{u}_{n}|^{2}+V_{n}(x) \tilde{u}^{2}_{n} \, dx\right)} g(\e_{n} x+\e_{n} \tilde{y}_{n}, \tilde{u}_{n})-V_{n}(x) \tilde{u}_{n}+\tilde{u}_{n}
$$
and
$$
V_{n}(x):=V(\e_{n} x+\e_{n} \tilde{y}_{n}).
$$
Put $\xi(x):=\frac{1}{M(\|\tilde{u}\|_{0}^{2})}f(\tilde{u})-V(y_{0})\tilde{u}+\tilde{u}$. By using Lemma \ref{Moser}, the interpolation in the $L^{p}$ spaces, $\tilde{u}_{n}\rightarrow \tilde{u}$ in $H^{s}(\R^{3})$, the assumptions $(g_1)$, $(g_3)$ and the continuity of $M$, we can see that 
$$
\xi_{n}\rightarrow \xi \mbox{ in } L^{p}(\R^{3}) \quad \forall p\in [2, \infty), 
$$
so, there exists $C>0$ such that 
$$
\|\xi_{n}\|_{L^{\infty}(\R^{3})}\leq C \mbox{ for any } n\in \mathbb{N}.
$$
Hence $\tilde{u}_{n}(x)=(\mathcal{K}*\xi_{n})(x)=\int_{\R^{3}} \mathcal{K}(x-z) \xi_{n}(z) \,dz$, where 
$\mathcal{K}$  is the Bessel kernel and satisfies the following properties \cite{FQT}:
\begin{compactenum}[$(i)$]
\item $\mathcal{K}$ is positive, radially symmetric and smooth in $\R^{3}\setminus \{0\}$,
\item there is $C>0$ such that $\displaystyle{\mathcal{K}(x)\leq \frac{C}{|x|^{3+2s}}}$ for any $x\in \R^{3}\setminus \{0\}$,
\item $\mathcal{K}\in L^{r}(\R^{3})$ for any $r\in [1, \frac{3}{3-2s})$.
\end{compactenum} 
Then, arguing as in Lemma $2.6$ in \cite{AM}, we can see that 
$$
\tilde{u}_{n}(x)\rightarrow 0 \mbox{ as } |x|\rightarrow \infty
$$
uniformly in $n\in \mathbb{N}$.
Therefore, there exists $R>0$ such that 
$$
\tilde{u}_{n}(x)<a \, \mbox{ for } \, |x|\geq R, n\in \mathbb{N}.
$$ 
Hence $u_{\e_{n}}(x)<a$ for any $x\in \R^{3}\setminus B_{R}(\tilde{y}_{n})$ and $n\in \mathbb{N}$. This fact and \eqref{ern}, show that there exists $\nu \in \mathbb{N}$ such that for any $n\geq \nu$ and $r/\e_{n}>R$ we have
$$
\R^{3}\setminus \Lambda_{\e_{n}}\subset \R^{3} \setminus B_{\frac{r}{\e_{n}}}(\tilde{y}_{n})\subset \R^{3}\setminus B_{R}(\tilde{y}_{n}),
$$
which implies that $u_{\e_{n}}(x)<a$ for any $x\in \R^{3}\setminus \Omega_{\e_{n}}$ and $n\geq \nu$. This gives a contradiction because of \eqref{eee}. \\
Let $\bar{\e}_{\delta}>0$ given by Theorem \ref{thm3.2}, and we fix $\e \in (0, \e_{\delta})$ where $\e_{\delta}= \min \{\tilde{\e}_{\delta}, \bar{\e}_{\delta}\}$. 
In view of Theorem \ref{thm3.2}, we know that the problem \eqref{Pea} admits at least $cat_{\Lambda_{\delta}}(\Lambda)$ nontrivial solutions. Let us denote by $u_{\e}$ one of these solutions. Since $u_{\e}\in \widetilde{\mathcal{N}}_{\e}$ satisfies \eqref{infty}, from the definition of $g$ it follows that $u_{\e}$ is a solution of \eqref{Pe}. Then $\hat{u}(x)=u(x/\e)$ is a solution to \eqref{P}, and we can conclude that \eqref{P} has at least $cat_{\Lambda_{\delta}}(\Lambda)$ solutions.\\ 
Finally, we study the behavior of the maximum points of solutions to the problem \eqref{Pe}. Take $\e_{n}\rightarrow 0$ and consider a sequence $\{u_{n}\}_{n\in \mathbb{N}}\subset \mathcal{H}_{\e_{n}}$ of solutions to \eqref{Pe}.
Let us observe that $(g_1)$ implies that we can find $\gamma>0$ such that 
\begin{equation}\label{4.4FS}
g(\e x, t)t\leq \frac{V_{0}}{K} t^{2} \mbox{ for any } x\in \R^{3}, t\leq \gamma.
\end{equation}
Arguing as before, we can find $R>0$ such that 
\begin{equation}\label{4.5FS}
\|u_{n}\|_{L^{\infty}(B^{c}_{R}(\tilde{y}_{n}))}<\gamma.
\end{equation}
Moreover, up to extract a subsequence, we may assume that 
\begin{equation}\label{4.6FS}
\|u_{n}\|_{L^{\infty}(B_{R}(\tilde{y}_{n}))}\geq \gamma.
\end{equation}
Indeed, if \eqref{4.6FS} does not hold, in view of \eqref{4.5FS} we can see that $\|u_{n}\|_{L^{\infty}(\R^{3})}<\gamma$. Then, by using $\langle \mathcal{J}'_{\e_{n}}(u_{n}), u_{n}\rangle=0$ and \eqref{4.4FS} we can infer
\begin{equation*}
m_{0} \|u_{n}\|_{\e_{n}}^{2}\leq \int_{\R^{3}} g(\e_{n} x, u_{n}) u_{n} \,dx\leq \frac{V_{0}}{K} \int_{\R^{3}} u_{n}^{2} \, dx
\end{equation*}
which yields $\|u_{n}\|_{\e_{n}}=0$, and this is impossible.
As a consequence, \eqref{4.6FS} holds. Taking into account \eqref{4.5FS} and \eqref{4.6FS} we can deduce that the maximum points $p_{n}\in \R^{3}$ of $u_{n}$ belong to $B_{R}(\tilde{y}_{n})$. Therefore, $p_{n}=\tilde{y}_{n}+q_{n}$ for some $q_{n}\in B_{R}(0)$. Hence, $\eta_{n}=\e_{n} \tilde{y}_{n}+\e_{n} q_{n}$ is the maximum point of $\hat{u}_{n}(x)=u_{n}(x/\e_{n})$. Since $|q_{n}|<R$ for any  $n\in \mathbb{N}$ and $\e_{n} \tilde{y}_{n}\rightarrow y_{0}\in \Lambda$ (in view of Lemma \ref{lem3.1N}), from the continuity of $V$ we can infer that
$$
\lim_{n\rightarrow \infty} V(\eta_{\e_{n}})=V(y_{0})=V_{0},
$$
which ends the proof of the Theorem.
\end{proof}

\smallskip

\noindent
{\bf Acknowledgements.} 
The authors warmly thank the anonymous referee for her/his useful and nice comments on the paper. 
The manuscript has been carried out under the auspices of the INDAM - Gnampa Project 2017 titled:{\it Teoria e modelli per problemi non locali}.

\end{document}